\documentclass[12pt]{amsart}
\usepackage{amsmath,amssymb,setspace}
\usepackage{mathrsfs}
\usepackage[active]{srcltx}
\usepackage[pagebackref, colorlinks, linkcolor=red, citecolor=blue, urlcolor=blue, hypertexnames=true]{hyperref}
\usepackage[backrefs]{amsrefs}
\usepackage{enumerate}
\allowdisplaybreaks
\usepackage[all,cmtip]{xy}

\usepackage{soul}
\usepackage{color,xcolor}

\newtheorem{thm}{Theorem}[section]
\newtheorem{condition}[thm]{Condition}

\newtheorem{lem}[thm]{Lemma}

\newtheorem{definition}[thm]{Definition}
\newtheorem{example}[thm]{Example}
\newtheorem{remark}[thm]{Remark}
\newtheorem{Proposition}[thm]{Proposition}

\theoremstyle{definition}

\makeatletter
\newcommand{\rmnum}[1]{\romannumeral #1}
\newcommand{\Rmnum}[1]{\expandafter\@slowromancap\romannumeral #1@}
\makeatother


\begin{document}
	\title[Metric mean dimension of free semigroup actions]{Metric mean dimension of free semigroup actions for non-compact sets}
	\author[Yanjie Tang, Xiaojiang Ye and Dongkui Ma]{}
	
	\email{yjtang1994@gmail.com}
	\email{yexiaojiang12@163.com}
	\email{dkma@scut.edu.cn}

	\date{\today}
	\thanks{{$^{*}$}Corresponding author: Dongkui Ma}
	
	\maketitle
	\centerline{\scshape Yanjie Tang$^1$, Xiaojiang Ye$^1$ and Dongkui Ma$^{*1}$}
	\medskip
	{\footnotesize
		\centerline{$^1$School of Mathematics, South China University of Technology, }
		\centerline{Guangzhou 510641, P.R. China}
	} 
	
	\hspace{2mm}
	
	\begin{abstract}
		In this paper, we introduce the notions of upper metric mean dimension, $u$-upper metric mean dimension, $l$-upper metric mean dimension  of free semigroup actions for non-compact sets via Carath\'{e}odory-Pesin structure. Firstly, the lower and upper estimations of the upper metric mean dimension of free semigroup actions are obtained by local metric mean dimensions. Secondly, one proves a variational principle that relates the $u$-upper metric mean dimension of free semigroup actions for non-compact sets with the corresponding skew product transformation. Furthermore, using the variational principle above, $\varphi$-irregular set acting on free semigroup actions shows full upper metric mean dimension in the system with the gluing orbit property. Our analysis generalizes the results obtained by Carvalho et al. \cite{MR4348410}, Lima and Varandas \cite{MR4308163}.
	\end{abstract}
	
	
	\section{Introduction}
	Topological entropy is a fundamental quantity used to measure the complexity of dynamical systems. Yano in \cite{MR579700} proved  that a closed manifold of dimension at least two the topological entropy is infinite for generic homeomorphisms. It is then a natural problem to distinguish the complexity of two systems with infinite topological entropy. In the late 1990s, Gromov \cite{MR1742309} proposed a new dynamical concept of dimension that was meant to extend the usual topological dimension to broader contexts. This notion, called mean topological dimension, is a topological invariant and defined for continuous maps on compact metric spaces in terms of the growth rate of refinements of coverings of the phase space, and is hard to compute in general. Further, Lindenstrauss and Weiss \cite{MR1749670} introduced the metric mean dimension to provide nontrivial information for infinite dimensional dynamical systems of infinite topological entropy and the well-known fact that it is an upper bound of mean topological dimension. Unlike the definition of topological entropy, the metric mean dimension depends on the selection of the metric. 
	
	\hspace{4mm}
	 It has several applications which cannot be touched within the framework of topological entropy  \cite{MR1749670,MR1793417,MR3939578,MR4025517,MR3763403, MR4308163,MR3798396}.
	For instance, Lima and Varandas  in \cite{MR4308163} considered homeomorphisms homotopic to the identity on the torus and employed precisely the metric mean dimension as the finer scaling of complexity they needed to describe the multifractal aspects of the sets of points with prescribed rotation vectors.
	Recently, Lindenstrauss and Tsukamoto's pioneering work  \cite{MR3990194} connected mean dimension to some information-theoretic quantity, called {Double Variational Principle}, which is similar to the classical variational principle in dynamical systems for topological entropy. 
	
	\hspace{4mm}
	Given a continuous map $f: X \rightarrow X$ on a compact metric space $(X, d)$ and a continuous observable $\varphi: X \rightarrow \mathbb{R}^d(d \geq 1)$, the set of points with $\varphi$-irregular is
	$$
	X_{\varphi, f}:=\left\{x \in X: \lim _{n \rightarrow \infty} \frac{1}{n} \sum_{i=0}^{n-1} \varphi\left(f^i(x)\right)\text{ does not exist}
	\right\}.
	$$
	The term ‘historic behavior’ was coined after some dynamics where the phenomenon of the persistence of points with this kind of behavior occurs \cite{MR2396607, MR1858471}. The irregular set is not detectable from the measure-theoretic viewpoint as the Birkhoff’s ergodic theorem ensures the irregular set has zero measure with respect to any invariant probability measure.  However, it is an increasingly well-known phenomenon that the irregular set can be large from the point of view of dimension theory. It was first proved by Pesin and Pitskel$^\prime$ \cite{MR775933} that in the case of full shift on two symbols the set $X_{\varphi, f}$ is either empty or carries full topological entropy. Furthermore, Barreira and Schmeling \cite{MR1759398} proved that for subshifts of finite type, conformal repellers and conformal horseshoes,  the set $X_{\varphi, f}$ carries full topological entropy and full Hausdorff dimension. There are lots of advanced results to show that the irregular set can carry full entropy with specification-like, shadowing-like, see \cite{MR2158401, MR3833343,MR2931333,MR1942414}. To obtain yet another mechanism to describe the topological complexity of the set of points with historic behavior and to pave the way to multifractal analysis, Lima and Varandas \cite{MR4308163} introduced the metric mean dimension for any non-compact subset using Carath\'eodory-Pesin structure (see \cite{MR1489237}), and they proved that under the gluing orbit property, 
	\begin{center}
		if  $X_{\varphi, f}\neq\emptyset$, then $\overline{\mathrm{mdim}}_{X_{\varphi, f}}(f)=\overline{\mathrm{mdim}}_{X}(f)$,
	\end{center}
	where $\overline{\mathrm{mdim}}_{X_{\varphi, f}}(f)$ denotes the metric mean dimension of $X_{\varphi, f}$ defined by Lima and Varandas \cite{MR4308163}.
	
	\hspace{4mm}
	People have become increasingly concerned with the research of free semigroup actions in recent years. On the one hand, it is needed by some other disciplines, such as physics, to allow the system that describes the real events to readjust over time to account for the inevitable experimental errors in \cite{MR2808288}. Some dynamic system theories, on the other hand, are closely related to it, such as the case of a foliation on a manifold and a pseudo-group of holonomy maps. The geometric entropy of finitely generated pseudogroup has been introduced \cite{MR926526} and shown to be a useful tool for studying the topology and dynamics of foliated manifolds. Metric mean dimension on the whole phase space of free semigroup actions was introduced by Carvalho et al. \cite{MR4348410} which proved a variational principle that relates the metric mean dimension of the semigroup action with the corresponding notions for the associated skew product and the shift map.
	
	\hspace{4mm}
	The above results raise the question of whether similar sets exist in dynamical systems of free semigroup actions. In order to do so, we introduce the notion of metric mean dimension of a free semigroup action for non-compact subsets.
	
	\hspace{4mm}
	This paper is organized as follows. In Sect. \ref{main results}, we give our main results. In Sect. \ref{Preliminaries}, we give some preliminaries. In Sect. \ref{Upper}, by using the Carath\'{e}odory-Pesin structure we give the new definitions of the upper metric mean dimension of free semigroup actions. Several of their properties are provided. In Sect. \ref{proofs}, we give the proofs of the main results.

\section{Statement of Main Results}\label{main results}
	Let $(X,d)$ and  $(Y,d_Y)$ be compact metric spaces, $f_y:X\to X$ be a continuous self-map for all $y\in Y$. Consider the free semigroup $(G,\circ)$ with generator $G_1:=\{f_y:y\in Y\}$ where the semigroup operation $\circ$ is the composition of maps.  In what follows, we will assume that the generator set $G_1$ is minimal, meaning that no transformation $f_y$, $y\in Y$ can be expressed as a composition of the remaining generators. Let $\mathcal{Y}$ be the set of all finite words formed by the elements of $Y$, that is, $\mathcal{Y}=\bigcup_{N\in\mathbb{N}}Y^N$. Obviously, $\mathcal{Y}$ with respect to the law of composition is a free semigroup generated by elements of $Y$ as generators. 
	
	\hspace{4mm}
	For convenience, we first recall the notion of words.
	For $w\in\mathcal{Y}$, we write $w^{\prime} \leq w$ if there exists a word $w^{\prime \prime} \in \mathcal{Y}$ such that $w=w^{\prime \prime} w^{\prime}$, $|w|$ stands for the length of $w$, that is, the number of symbols in $w$. If $\omega=(i_1i_2\cdots) \in Y^{\mathbb{N}}$, and $a, b \in \mathbb{N}$ with $a \leq b$, write $\omega|_{[a, b]}=w$ if $w=i_a i_{a+1} \cdots i_b$. Notice that $\emptyset \in \mathcal{Y}$ and $\emptyset \leq w$. For $w=i_1 i_2 \cdots i_k \in \mathcal{Y}$, denote $\overline{w}=i_k \cdots i_2 i_1$. For $w \in \mathcal{Y}$, $w=i_1 \cdots i_{k}$, let us write $f_w=f_{i_1} \circ \cdots \circ f_{i_{k}}$. Note that if $k=0$, that is, $w=\emptyset$, define $f_w=\operatorname{Id}$, where Id is the identity map. Obviously, $f_{w w^{\prime}}=f_w f_{w^{\prime}}$. We set $f_w^{-1}=(f_w)^{-1}$ for $w\in\mathcal{Y}$.
	
	\hspace{4mm}
	Our first main result is an estimate of the upper metric mean dimension using local metric mean dimensions inspired by Ma and Wen \cite{MR2412786} and Ju et al. \cite{MR3918203}. Let $\mathcal{M}(X)$ denote the set of all Borel probability measures on $X$. For $x\in X$ and $w\in\mathcal{Y}$, denote $B_w(x, \varepsilon)$ the $(w, \varepsilon)$-Bowen ball at $x$. Inspired by Ju et al. \cite{MR3918203}, we introduce the concepts of lower and upper local entropies of free semigroup actions as follows. For $\mu\in\mathcal{M}(X)$, 
	$$
	h_{\mu, G}^{L^{+}}(x):=\lim _{\varepsilon \rightarrow 0}h_{\mu, G}^{L^{+}}(x,\varepsilon),
	$$
	where
	$$
	h_{\mu, G}^{L^{+}}(x,\varepsilon):=\liminf _{n \rightarrow \infty}-\frac{1}{n+1} \log \inf _{w\in Y^n}\left\{\mu\left(B_w(x, \varepsilon)\right)\right\},
	$$
	is called the $L^{+}$ lower local entropy of $\mu$ at point $x$ with respect to $G$, while the quantity
	$$
	h_{\mu, G}^{L^{-}}(x):=\lim _{\varepsilon \rightarrow 0} h_{\mu, G}^{L^{-}}(x,\varepsilon),
	$$
	where
	$$
	h_{\mu, G}^{L^{-}}(x,\varepsilon):=\liminf _{n \rightarrow \infty}-\frac{1}{n+1} \log \sup _{w\in Y^n}\left\{\mu\left(B_w(x, \varepsilon)\right)\right\},
	$$
	is called the $L^{-}$ lower local entropy of $\mu$ at point $x$ with respect to $G$.
	\begin{remark}
		If $\sharp Y=m$, that is, $G_1=\{f_0,f_1,\cdots,f_{m-1}\}$, then $h_{\mu, G}^{L^{+}}(x)$ and $h_{\mu, G}^{L^{-}}(x)$ coincide with $L^+$ and $L^{-}$ lower local entropy of $\mu$ at point $x$ with respect to $G$ respectively defined by Ju et al. \cite{MR3918203}.
		If $\sharp Y=1$, that is $G_1=\{f\}$, then $h_{\mu, G}^{L^{+}}(x)=h_{\mu, G}^{L^{-}}(x)$, i.e., the lower local entropy for $f$ defined by Brin and Katok \cite{MR730261}.
	\end{remark}
	
	\hspace{4mm}
	In order to have a concept related to the metric mean dimension, we introduced the following concepts.
	\begin{definition}
		For $\mu\in\mathcal{M}(X)$, we define the $L^{+}$ upper local metric mean dimension as
		$$
		\overline{\mathrm{mdim}}_\mu (x,G):=\limsup_{\varepsilon\to 0}\frac{h_{\mu, G}^{L^{+}}(x,\varepsilon)}{\log\frac{1}{\varepsilon}},
		$$
		and define the $L^{-}$ lower local metric mean dimension as
		$$
		\underline{\mathrm{mdim}}_\mu (x,G):=\liminf_{\varepsilon\to 0}\frac{h_{\mu, G}^{L^{-}}(x,\varepsilon)}{\log\frac{1}{\varepsilon}}.
		$$
	\end{definition}
	
	\hspace{4mm}
	Now we give two estimations about the upper metric mean dimension of free semigroup action on $Z\subseteq X$:
	\begin{thm} \label{theorem1}
		Let  $\mu$ be a Borel probability measure on $X$, $Z$ a Borel subset of $X$ and $s\in (0,\infty)$. 
		\begin{enumerate}[(i)]
			\item \label{theorem5.1}
			If $\underline{\mathrm{mdim}}_\mu (x,G)\ge s$ for all $x\in Z$ and $\mu (Z)>0$ then $\overline{\mathrm{mdim}}_Z(G,d,\mathbb{P})\ge s$.
			\item \label{theorem3}
			If $\overline{\mathrm{mdim}}_\mu (x,G)\le s$ for all $x\in Z$ then $\overline{\mathrm{mdim}}_Z(G,d,\mathbb{P})\le s$.
		\end{enumerate}
	Here $\mathbb{P}$ is a random walk on $Y^{\mathbb{N}}$, $\overline{\mathrm{mdim}}_Z(G,d,\mathbb{P})$ denotes the upper metric mean dimension of free semigroup action $G$ with respect to $\mathbb{P}$ on the set $Z$ (see Sec. \ref{Upper}).
	\end{thm}
	
	\hspace{4mm}
	Next, the second result describes a variational principle that relates the metric mean dimension of the semigroup action for non-compact sets with the corresponding notions for the associated skew product on $Y^\mathrm{N}\times X$, and compares them with the upper box dimension of $Y$. For $\nu\in \mathcal{M}(Y)$, denote $\mathrm{supp}\nu$ the support of $\nu$ on $Y$. Let $F: Y^\mathrm{N}\times X\to Y^\mathrm{N}\times X$ be the skew product transformation, $D$ be the product metric on $Y^\mathrm{N}\times X$. $\overline{\mathrm{dim}}_BY$ denotes the upper box dimension of $(Y,d_Y)$,  $\mathcal{H}_Y$ the set of such homogeneous Borel probability measures on $Y$, $\overline{\mathrm{umdim}}_{Y^{\mathbb{N}} \times Z}\left(F, D\right)$ the $u$-upper metric mean dimension with 0 potential of $F$  on the set $Y^{\mathbb{N}} \times Z$ (see \cite{MR4216094}). Then we have the following theorem:
	\begin{thm}\label{theorem4}
		For any subset $Z\subset X$, if $\overline{\mathrm{dim}}_BY<\infty$ and $\nu\in \mathcal{H}_Y$, then
		\begin{enumerate}[(i)]
			\item \label{theorem4-1}
			$\overline{\mathrm{dim}}_B(\mathrm{supp}\nu)+\overline{\mathrm{umdim}}_Z \left(G, d, \nu ^\mathbb{N}\right) \le \overline{\mathrm{umdim}}_{Y^{\mathbb{N}}\times Z}\left( F, D\right)$;
			\item \label{theorem4-2}
			if, $\mathrm{supp} \nu=Y$,
			\begin{equation}\label{equation5}
				\overline{\mathrm{dim}}_B Y+\overline{\mathrm{umdim}}_Z \left(G, d, \nu ^\mathbb{N} \right)= \overline{\mathrm{umdim}}_{Y^{\mathbb{N}} \times Z}\left(F, D\right).
			\end{equation}
		\end{enumerate}
	Here $\nu^\mathbb{N}$ denotes the product measure on $Y^\mathbb{N}$,  $\overline{\mathrm{umdim}}_Z \left(G, d, \nu ^\mathbb{N} \right)$ denotes the $u$-upper metric mean dimension of free semigroup action $G$ with respect to $\nu ^\mathbb{N}$ on the set $Z$ (see Sec. \ref{Upper}).
	\end{thm}
	\begin{remark}
		\item [(\rmnum{1})]
		If $Z=X$, Theorem \ref{theorem4} generalizes the result obtained by Carvalho et al. \cite{MR4348410}.
		\item [(\rmnum{2})]
		If $\sharp Y=m$ and $\mathbb{P}$ is generated by the probability vector $\nu:=\left(p_1, \cdots, p_m\right)$ with $\sum_{i=1}^mp_i=1$ and $p_i >0$ for all $i=1,\cdots,m$, it follows from Theorem \ref{theorem4} that 
		$$
		\overline{\mathrm{umdim}}_Z \left(G, d, \nu ^\mathbb{N}\right)=\overline{\mathrm{umdim}}_{Y^{\mathbb{N}} \times Z}\left(F, D\right).
		$$
	\end{remark}

	\hspace{4mm}
	Finally, the third result is that $\varphi$-irregular set of free semigroup actions carries full upper metric mean dimension using Theorem \ref{theorem4}. The irregular set arises in the context of multifractal analysis. As a consequence of Birkhoff’s ergodic theorem, the irregular set is not detectable from the point of view of an invariant measure.
           Let  $\varphi: X \rightarrow \mathbb{R}^d$ be a continuous function. Recall that a point $x \in X$ is called to be $\varphi$-irregular point of free semigroup action $G$ if there exists $\omega\in Y^{\mathbb{N}}$, the limit $\lim _{n \rightarrow \infty} \frac{1}{n} \sum_{j=0}^{n-1} \varphi\left(f_{\overline{\omega|_{[1,j]}}}(x)\right)$ does not exist, which was introduced by Zhu and Ma \cite{MR4200965}. Let $I_\varphi(G)$ denote the set of all $\varphi$-irregular points of free semigroup action, that is,
	$$
	I_\varphi(G):=\left\{x\in X: \lim_{n\to\infty}\frac{1}{n}\sum_{j=0}^{n-1}\varphi\left(f_{\overline{\omega|_{[1,j]}}}(x)\right)\text{ does not exist for some }\omega\in Y^{\mathbb{N}}\right\}.
	$$
	\begin{thm}\label{theorem5}
		Suppose that $G$ has the gluing orbit property,  $\overline{\mathrm{dim}}_B Y<\infty$ and $\nu$ is a homogeneous measure with $\mathrm{supp} \nu=Y$. Let $\varphi: X \rightarrow \mathbb{R}^d$ be a continuous function. If $I_\varphi(G)\neq\emptyset$, then
		$$
		\overline{\mathrm{umdim}}_{ I_{\varphi}(G)}(G,d,\mathbb{P})=\overline{\mathrm{mdim}}_M(X,G,d, \mathbb{P})=  \overline{\mathrm{umdim}}_{ X}(G,d,\mathbb{P}),
		$$
		where $\overline{\mathrm{mdim}}_M(X,G,d, \mathbb{P})$ denotes the upper metric mean dimension of free semigroup action $G$ with respect to $\mathbb{P}$ on the whole phase defined by Carvalho et al. \cite{MR4348410}.
	\end{thm}
\begin{remark}
	When $\sharp Y=1$, that is , $G_1=\{f\}$, Theorem \ref{theorem5} coincides with the result obtained by Lima and Varandas \cite{MR4308163}.
\end{remark}

	\section{Preliminaries}\label{Preliminaries}
	We start recalling the main concepts we use and describing the systems we will work with.
	
\subsection{Metric mean dimension on the whole space of free semigroup actions }\label{section2.1}
Let $(X,d)$ and  $(Y,d_Y)$ be compact metric spaces, $f_y:X\to X$ be a continuous self-map for all $y\in Y$, $G$ be the free semigroup acting on $X$ generated by $G_1=\{f_y:y\in Y\}$.
	
	\hspace{4mm}
	A random walk $\mathbb{P}$ on $Y^{\mathbb{N}}$ is a Borel probability measure in this space of sequences which is invariant by the shift map $\sigma$. For instance, we may consider a finite subset $E=\left\{y_0,y_1, \cdots, y_{m-1}\right\}\subset Y$, a probability vector $\mathbf{p}:=(p_0,p_1,\cdots, p_{m-1})$ with $p_i>0$ and $\sum_{i=0}^{m-1} p_i=1$, the probability measure $\nu:=\sum_{i=0}^{m-1} p_i \delta_{y_i}$ on $E$ and the Borel product measure $\nu^{\mathbb{N}}$ on $Y^{\mathbb{N}}$. Such Borel product measure $\nu^{\mathbb{N}}$ will be called a Bernoulli measure, which is said to be symmetric if $p_i=(1 / k)$ for every $i \in\{0,1, \ldots, m-1\}$. If $Y$ is a Lie group, a natural symmetric random walk is given by $\nu^{\mathbb{N}}$, where $\nu$ is the Haar measure. 
	
	\hspace{4mm}
	For $w \in \mathcal{Y}$, we assign a metric $d_w$ on $X$ by setting
	$$
	d_w\left(x_1, x_2\right):=\max _{w^{\prime} \leq \overline{w}} d\left(f_{w^{\prime}}\left(x_1\right), f_{w^{\prime}}\left(x_2\right)\right).
	$$
	Given a number $\delta>0$ and a point $x \in X$, define the $(w, \delta)$-Bowen ball at $x$ by
	$$
	B_w(x, \delta):=\left\{y \in X: d_w(x, y)<\delta\right\}.
	$$
	
	\hspace{4mm}
	Restate that the separated set and spanning set of free semigroup actions were introduced by Bufetov \cite{MR1681003}. For $Z\subset X$, $w\in \mathcal{Y}$ and $\varepsilon>0$, a subset $K\subset Z$ is called a $(w, \varepsilon, Z, G)$-separated set of $Z$ if, for any $x_1, x_2 \in K$ with $x_1 \neq x_2$, one has $d_w\left(x_1, x_2\right) > \varepsilon$. The maximum cardinality of a $(w, \varepsilon, Z, G)$-separated subset of $Z$ is denoted by $s(w, \varepsilon,Z, G)$. A subset $E \subset Z$ is said to be $(w, \varepsilon,G)$-spanning set of $Z$ if for every $x \in X$ there is $y \in E$ such that $d_w(x,y)\le\varepsilon$. The smallest cardinality of any $(w, \varepsilon, Z, G)$-spanning subset of $Z$ is denoted by $r(w, \varepsilon,Z, G)$.
	
	\hspace{4mm}
	Obviously, 
	\begin{equation}\label{equation3}
		r(w,\varepsilon,Z, G)\le s(w,\varepsilon,Z, G)\le r(w,\frac{\varepsilon}{2},Z, G).
	\end{equation}
	
	\hspace{4mm}
	We recall the definition of metric mean dimension on the whole phase space of free semigroup actions introduced by Carvalho et al. \cite{MR4348410}. Let $\mathbb{P}$ be a product measure on $Y^{\mathbb{N}}$ generated by Borel probability measure $\nu$ supported on $Y$.
	The topological entropy of free semigroup action $G$ is given by
	$$
	h_X(G,\mathbb{P}):=\lim_{\varepsilon\to 0}\limsup_{n\to\infty}\frac{1}{n} \log \int_{Y^{\mathbb{N}}} s\left(\omega|_{[1,n]}, \varepsilon,X, G\right) d \mathbb{P}(\omega),
	$$
	which was introduced by Carvalho et al. \cite{MR3784991}.  
	\begin{remark}
		If  $\sharp Y=m$ and the probability measure $\mathbb{P}$ is generated by $(\frac{1}{m},\cdots,\frac{1}{m})$, then $h_X(G,\mathbb{P})$ coincides with the definition of topological entropy introduced by Bufetov \cite{MR1681003}.
	\end{remark}
	
	\begin{definition}\cite{MR4348410}
		The upper and lower metric mean dimensions of the free semigroup action $G$ with respect to $\mathbb{P}$ are given respectively by
		$$
		\begin{aligned}
			& \overline{\operatorname{mdim}}_M(X, G, d, \mathbb{P})=\limsup _{\varepsilon \rightarrow 0} \frac{h(X, G, \mathbb{P}, \varepsilon)}{-\log \varepsilon}, \\
			& \underline{\operatorname{mdim}}_M(X, G, d, \mathbb{P})=\liminf _{\varepsilon \rightarrow 0} \frac{h(X, G, \mathbb{P}, \varepsilon)}{-\log \varepsilon},
		\end{aligned}
		$$
		where
		$$
		h(X, G, \mathbb{P}, \varepsilon)=\limsup _{n \rightarrow \infty} \frac{1}{n} \log \int_{Y^{\mathbb{N}}} s\left(\omega|_{[1,n]}, \varepsilon,X, G\right) d \mathbb{P}(\omega).
		$$
	\end{definition}
	\begin{remark}
		If $\sharp Y=1$,  these definitions coincide with the upper and lower metric mean dimensions of a single map on the whole phase defined by Lindenstrauss and Weiss \cite{MR1749670}.
	\end{remark}
	
	\hspace{4mm}
	The dynamical systems given by free semigroup action have a strong connection with the skew product which has been analyzed to obtain properties of free semigroup actions through fiber associated with the skew product (see for instance \cite{MR3784991, MR4200965,MR4348410}). Recall that the skew product transformation is given as follows:
	$$
	F: Y^{\mathbb{N}} \times X \rightarrow Y^{\mathbb{N}} \times X,(\omega, x) \mapsto\left(\sigma(\omega), f_{i_1}(x)\right),
	$$
	where $\omega=\left(i_1, i_2, \cdots\right)$ and $\sigma$ is the shift map of $Y^{\mathbb{N}}$. The metric $d^\prime$ on $Y^{\mathbb{N}}$ is given by 
	$$
	d^\prime(\omega,\omega^\prime):=\sum_{j=1}^\infty\frac{d_Y(i_j,i_j^\prime)}{2^j}.
	$$
	The metric $D$ on $Y^{\mathbb{N}} \times X$ is given by the formula
	$$
	D\left((\omega, x),\left(\omega^{\prime}, x^{\prime}\right)\right):=\max \left\{d^{\prime}\left(\omega, \omega^{\prime}\right), d\left(x, x^{\prime}\right)\right\}.
	$$
	
		\hspace{4mm}
	The specification property of free semigroup actions was introduced by Rodrigues and Varandas \cite{MR3503951}. 
	\begin{definition}\label{specification}
		\cite{MR3503951} We say that $G$ has the specification property if for any $\varepsilon>0$, there exists $\mathfrak{m}(\varepsilon)>0$, such that for any $k>0$, any points $x_{1},\cdots, x_{k} \in X$, any positive integers $n_{1}, \cdots, n_{k}$, any word $w_{(n_{1})} \in Y^{n_1}$,$\cdots$, $w_{(n_{k})} \in Y^{n_k}$, any $p_{1},\cdots, p_{k} \geq \mathfrak{m}(\varepsilon)$, any $w_{(p_{1})} \in Y^{p_1}$, $\cdots$, $w_{(p_{k})} \in Y^{p_k}$,  one has
		$$
		B_{w_{(n_1)}}\left (x_1,\varepsilon\right )\cap\left (\bigcap_{j=2}^k {f^{-1}_{\overline{w_{(p_{j-1})}}\, \overline{w_{(n_{j-1})}}\cdots\overline{w_{(p_1)}}\, \overline{w_{(n_1)}}}}B_{w_{(n_j)}}\left (x_j,\varepsilon\right )\right )\neq\emptyset.
		$$
	\end{definition}

	\hspace{4mm}
	If $m=1$,  the specification property of free semigroup actions coincides with the classical definition introduced by Bowen \cite{MR282372}.
	
	\subsection{Some concepts}
	Let $(Y,d_Y)$ be a compact metric space, $\nu$ be a Borel probability measure on $Y$. A balanced measure should give the same probability to any two balls with the same radius, but in general this is too strong a requirement. Bowen \cite{MR0274707} therefore introduced a definition of the chi-square measure. In this paper, we only need the following definition which is weaker than Bowen's \cite{MR0274707}. 
	
	\begin{definition}
		\cite{MR0274707}
		We say that $\nu$ is homogeneous if there exists $L>0$ such that
		$$
		\nu \left(B\left(y_1, 2 \varepsilon\right)\right) \leq L \nu\left(B\left(y_2, \varepsilon\right)\right) \quad \forall y_1, y_2 \in \operatorname{supp} \nu \quad \forall \varepsilon>0 .
		$$
	\end{definition}
	
	\hspace{4mm}
	For instance, the Lebesgue measure on $[0,1]$, atomic measures, and probability measures absolutely continuous with respect to the latter ones, with densities bounded away from zero and infinity, are examples of homogeneous probability measures. We denote by $\mathcal{H}_Y$ the set of such homogeneous Borel probability measures on $Y$. For a discussion on conditions on $Y$ which ensure the existence of homogeneous measures, we refer the reader to (\cite{MR3137474}, Sec. 4) and references therein.
	
	\hspace{4mm}
	Next, we recall the definition of upper box dimension, see e.g. \cite{MR2118797} for more details.
	\begin{definition}
		\cite{MR2118797}
		The upper box dimension of $(Y,d_Y)$ is given by 
		$$
		\overline{\mathrm{dim}}_B Y:=\limsup_{\varepsilon\to 0} \frac{\log N_{d_Y}(\varepsilon)}{\log\frac{1}{\varepsilon}},
		$$
		where $N_{d_Y}(\varepsilon)$ denotes the maximal cardinality of $\varepsilon$-separated set of $(Y,d_Y)$.
	\end{definition}

	\hspace{4mm}
	The gluing orbit property was introduced in \cite{MR2921897} (with the terminology of transitive specification property) and independently in \cite{MR3944271} for homeomorphisms and flows. It bridges between completely non-hyperbolic dynamics (equicontinuous and minimal dynamics \cite{MR3603272,MR3960495}) and uniformly hyperbolic dynamics (see e.g. \cite{MR3944271}). Both of these properties imply a rich structure on the dynamics (see e.g. \cite{MR3603272, MR4125519}).
	\begin{definition}
		\cite{MR2921897}
		Let $(X,d)$ be a compact metric space, $f: X\to X$ a continuous self-map. We say that $f$ satisfies the gluing orbit property if for any $\varepsilon>0$, there exists an integer $\mathfrak{p}(\varepsilon) \geq 1$, so that for any points $x_1, x_2, \cdots, x_k \in X$, any positive integers $n_1, \cdots, n_k$, there are $0 \leq p_1, \cdots, p_{k-1} \leq \mathfrak{p}(\varepsilon)$ and a point $y \in X$ hold
		$$
		B_{n_1}(x_1,\varepsilon)\cap\bigcap_{j=2}^k f^{-(n_1+p_1+\cdots+n_{j-1}+p_{j-1})}\left(B_{n_j}(x_j,\varepsilon)\right)\neq\emptyset.
		$$
		Here $B_n(x,\varepsilon)$ denotes the $(n,\varepsilon)$-Bowen ball of $f$.
	\end{definition}
	
	\hspace{4mm}
	It is not hard to check that irrational rotations satisfy the gluing orbit property \cite{MR3603272}, but fail to satisfy the shadowing or specification properties. Partially hyperbolic examples exhibiting the same kind of behavior have been constructed in \cite{MR4159675}.

	\hspace{4mm}
	Under the gluing orbit property, the metric mean dimension of the irregular set has been studied in Lima and Varandas \cite{MR4308163}, but the metric mean dimension of such set has not been studied in dynamical systems of free semigroup actions. In this paper, we focus on the metric mean dimension of such set of free semigroup actions and obtain more extensive results. Therefore, it is important and necessary to introduce the gluing orbit property of free semigroup actions.

	\hspace{4mm}
Next, we introduce the concept of the gluing orbit property of free semigroup actions:
\begin{definition}\label{definition5}
	We say that $G$ satisfies the gluing orbit property, if for any $\varepsilon>0$, there exists $\mathfrak{p}(\varepsilon)>0$, 
	such that for any $k\ge 2$, any points $x_1,\cdots,x_k\in X$, any positive integers $n_1,\cdots,n_k$, any words $w_{(n_1)}\in Y^{n_1},\cdots, w_{(n_k)}\in Y^{n_k}$, there exist $0\le p_1,\cdots,p_{k-1}\le \mathfrak{p}(\varepsilon)$, such that for any words $w_{(p_1)}\in Y^{p_1},\cdots,w_{(p_{k-1})}\in Y^{p_{k-1}}$, one has
	$$
	B_{w_{(n_1)}}\left (x_1,\varepsilon\right )\cap\left (\bigcap_{j=2}^k {f^{-1}_{\overline{w_{(p_{j-1})}}\, \overline{w_{(n_{j-1})}}\cdots\overline{w_{(p_1)}}\, \overline{w_{(n_1)}}}}B_{w_{(n_j)}}\left (x_j,\varepsilon\right )\right )\neq\emptyset.
	$$
\end{definition}
	\begin{remark}
		It is clear that the specification property (see Definition \ref{specification}) implies the gluing orbit property for free semigroup actions. If $m = 1$, the gluing orbit property of free semigroup actions coincides with the definition of a single map introduced by Bomfim and Varandas \cite{MR3944271}.
	\end{remark}

	\hspace{4mm}
	We describe an example to help us interpret the gluing orbit property of free semigroup actions.

	\begin{example}
		Let $M$ be a compact Riemannian manifold, $Y=\{0,1,\cdots,m-1\}$ and $G$ the free semigroup generated by $G_1=\{f_0, \cdots, f_{m-1}\}$ on $M$ which are $C^1$-local diffeomorphisms such that for any $j=0, \cdots, m-1$, $\left\|D f_j(x) v\right\| \geq \lambda_j\|v\|$ for all $x \in M$ and all $v \in T_x M$, where $\lambda_j$ is a constant larger than 1. It follows from Theorem 16 of \cite{MR3503951} that $G$ satisfies specification property. Given $\varepsilon>0$, let $\mathfrak{p}(\varepsilon):=2\mathfrak{m}(\varepsilon)$ where $\mathfrak{m}(\varepsilon)$ is the positive integer in the definition of specification property of $G$ (see Definition \ref{specification}). For any points $x_1,\cdots,x_k\in M$, any positive integers $n_1,\cdots,n_k$, any words $w_{(n_1)}\in Y^{n_1},\cdots, w_{(n_k)}\in Y^{n_k}$, pick $p_1=p_2=\cdots =p_{k-1}=\mathfrak{m}(\varepsilon)$,
		for any words $w_{(p_1)}\in Y^{p_1},\cdots,w_{(p_{k-1})}\in Y^{p_{k-1}}$, by specification property it holds that
		$$
		B_{w_{(n_1)}}\left (x_1,\varepsilon\right )\cap\left (\bigcap_{j=2}^k {f^{-1}_{\overline{w_{(p_{j-1})}}\, \overline{w_{(n_{j-1})}}\cdots\overline{w_{(p_1)}}\, \overline{w_{(n_1)}}}}B_{w_{(n_j)}}\left (x_j,\varepsilon\right )\right )\neq\emptyset.
		$$
		Hence, $G$ has the gluing orbit property.
	\end{example}

\section{Upper metric mean dimension of free semigroup actions for non-compact sets and properties} \label{Upper}
	In this section, using Carath\'{e}odory-Pesin structure, we introduce the definitions of  upper metric mean dimension, $u$-upper metric mean dimension, $l$-upper metric mean dimension  of free semigroup action $G$ with respect to $\mathbb{P}$ for non-compact sets by open covers and Bowen's balls, respectively,  and provide some properties of them.
	
\subsection{Upper metric mean dimension of free semigroup actions for non-compact sets using open covers}
	For $0<\varepsilon<1$, let $C(\varepsilon)$ be the set of all the open covers of $X$ with diameter less than $\varepsilon$. Consider an open cover $\mathcal{U}$ of $X$ and denote by $\mathcal{W}_{n+1}(\mathcal{U})$ the collection of all strings $\mathbf{U}=(U_0, \cdots, U_{n})$ with length $\mathfrak{l}(\mathbf{U})=n+1$ where $U_j \in \mathcal{U}$ for all $0\le j\le n$. We put the Cartesian product 
	$$
	\mathcal{S}_{n+1}(\mathcal{U}):=\left\{ \left(w_{\mathbf{U}}, \mathbf{U}\right): \mathbf{U}\in\mathcal{W}_{n+1}(\mathcal{U}), \, w_{\mathbf{U}}\in Y^n\right\},
	$$
	and $\mathcal{S}(\mathcal{U}):=\bigcup_{n \ge 1}\mathcal{S}_n(\mathcal{U})$.
	
	\hspace{4mm}
	For $(w_{\mathbf{U}}, \mathbf{U})\in\mathcal{S}_{n+1}$, $w_{\mathbf{U}}=i_1i_2\cdots i_n$, we associate the set 
	
	$$
	\begin{aligned}
		X_{w_{\mathbf{U}}}(\mathbf{U}) :& =\left\{x \in X: x\in U_0,\,f_{i_1}(x)\in U_1,\cdots, f_{i_n\cdots i_1}(x)\in U_n\right\} \\
		& =U_0 \cap\left(f_{i_1}\right)^{-1}\left(U_1\right) \cap \cdots \cap\left(f_{i_n\cdots i_1}\right)^{-1}\left(U_n\right).
	\end{aligned}
	$$
	
	\hspace{4mm}
	The theory of Carath\'eodory dimension characteristic ensures the following definitions. Fixed $N\in\mathbb{N}$, $w\in Y^N$, $\lambda \in \mathbb{R}$, $Z \subset X$ and $0<\varepsilon<1$, we set
	$$
	M_w(Z, \lambda, N, \varepsilon,G, d):=\inf _{\mathcal{U} \in C(\varepsilon)} \inf _{\mathcal{G}_w(\mathcal{U})}\left\{\sum_{(w_{\mathbf{U}}, \mathbf{U}) \in \mathcal{G}_w(\mathcal{U})} e^ { -\lambda \mathfrak{l}(\mathbf{U})} \right\},
	$$
	where the second infimum is taken over finite or countable collections of strings $\mathcal{G}_w(\mathcal{U}) \subset \mathcal{S}(\mathcal{U})$ such that $\mathfrak{l}(\mathbf{U})\ge N+1$ and $w_{\mathbf{U}}|_{[1,N]}=w$ for all $(w_{\mathbf{U}}, \mathbf{U})\in \mathcal{G}_w(\mathcal{U})$ and 
	$
	Z \subset \bigcup_{(w_{\mathbf{U}}, \mathbf{U})\in \mathcal{G}_w(\mathcal{U})} X_{w_{\mathbf{U}}}(\mathbf{U}).
	$
	
	\hspace{4mm}
	For $\omega\in Y^{\mathbb{N}}$, put $w(\omega):=\omega|_{[1,N]}$, we define
	$$
	M(Z, \lambda, N, \varepsilon,G, d,\mathbb{P}):=\int_{ Y^{\mathbb{N}}}  M_{w(\omega)}(Z, \lambda, N, \varepsilon,G, d) d\mathbb{P}(\omega).
	$$
	Moreover, the function $N \mapsto M(Z, \lambda, N, \varepsilon,G, d,\mathbb{P})$ is non-decreasing as $N$ increases. Therefore, the following limit exists
	$$
	m(Z, \lambda, \varepsilon,G, d,\mathbb{P}):=\lim _{N \rightarrow+\infty} M(Z, \lambda, N, \varepsilon,G, d,\mathbb{P}).
	$$
	
	\hspace{4mm}
	Similarly, we define
	$$
	\begin{aligned}
		R_w(Z, \lambda, N, \varepsilon,G, d):&=\inf _{\mathcal{U} \in C(\varepsilon)} \inf _{\mathcal{G}_w(\mathcal{U})}\left\{\sum_{(w_{\mathbf{U}}, \mathbf{U}) \in \mathcal{G}_w(\mathcal{U})} e^{-\lambda (N+1)}\right\}\\
		&=e^{-\lambda (N+1)} \Lambda_w(Z, N, \varepsilon,G, d),
	\end{aligned}
	$$
	where $
	\Lambda_w(Z, N, \varepsilon,G, d):=\inf _{\mathcal{U} \in C(\varepsilon)} \inf _{\mathcal{G}_w(\mathcal{U})} \left \{ \sharp  \mathcal{G}_w(\mathcal{U})\right \},$
	and the second infimum is taken over finite or countable collections of strings $\mathcal{G}_w(\mathcal{U}) \subset \mathcal{S}(\mathcal{U})$ such that $\mathfrak{l}(\mathbf{U})= N+1$ and $w_{\mathbf{U}}=w$ for all $(w_{\mathbf{U}}, \mathbf{U})\in \mathcal{G}_w(\mathcal{U})$ and 
	$
	Z \subset \bigcup_{(w_{\mathbf{U}}, \mathbf{U}) \in \mathcal{G}_w(\mathcal{U})} X_{w_{\mathbf{U}}}(\mathbf{U}).
	$
	
	\hspace{4mm}
	Let
	$$
	\begin{aligned}
		R(Z, \lambda, N, \varepsilon,G, d,\mathbb{P}):&=\int_{ Y^{\mathbb{N}}}R_{w(\omega)}(Z, \lambda, N, \varepsilon,G, d,\mathbb{P}) d\mathbb{P}(\omega)\\
		&=e^{-\lambda (N+1)}\Lambda(Z, N, \varepsilon,G, d,\mathbb{P}), 
	\end{aligned}
	$$
	where $\Lambda(Z, N, \varepsilon,G, d,\mathbb{P})=\int_{ Y^{\mathbb{N}}}\Lambda_{w(\omega)}(Z, N, \varepsilon,G, d)d\mathbb{P}(\omega).$ We set
	$$
	\begin{aligned}
		\overline{r}(Z, \lambda, \varepsilon,G, d,\mathbb{P}):&=\limsup _{N \rightarrow+\infty} R(Z, \lambda, N, \varepsilon,G, d,\mathbb{P}), \\
		\underline{r}(Z, \lambda, \varepsilon,G, d,\mathbb{P}):&=\liminf _{N \rightarrow+\infty} R(Z, \lambda, N, \varepsilon,G, d,\mathbb{P}).
	\end{aligned}
	$$
	
	\hspace{4mm}
	When $\lambda$ goes from $-\infty$ to $+\infty$ the $m(Z, \lambda, \varepsilon,G, d)$,  $\overline{r}(Z, \lambda, \varepsilon,G, d)$, $\underline{r}(Z, \lambda, \varepsilon,G, d)$, jump from $+\infty$ to 0 at a unique critical value. We denote the critical values respectively as
	$$
	\begin{aligned}
		\overline{\mathrm{mdim}}_Z(\varepsilon,G, d,\mathbb{P}) :&=\inf \left \{\lambda:m(Z, \lambda, \varepsilon,G, d,\mathbb{P})=0\right \} \\
		&=\sup\left \{\lambda: m(Z, \lambda, \varepsilon,G, d,\mathbb{P})=\infty\right \},\\
		\overline{\mathrm{umdim}}_Z(\varepsilon,G, d,\mathbb{P}) :&=\inf \left \{\lambda:\overline{r}(Z, \lambda, \varepsilon,G, d,\mathbb{P})=0\right \} \\
		&=\sup\left \{\lambda:\overline{r}(Z, \lambda, \varepsilon,G, d,\mathbb{P})=\infty\right \},\\
		\underline{\mathrm{lmdim}}_Z(\varepsilon,G, d,\mathbb{P}) :&=\inf \left \{\lambda:\underline{r}(Z, \lambda, \varepsilon,G, d,\mathbb{P})=0\right \} \\
		&=\sup\left \{\lambda:\underline{r}(Z, \lambda, \varepsilon,G, d,\mathbb{P})=\infty\right \}.
	\end{aligned}
	$$
	Put
	$$
	\begin{aligned}
		\overline{\mathrm{mdim}}_Z(G, d,\mathbb{P})&:=\limsup _{\varepsilon \rightarrow 0}\frac{\overline{\mathrm{mdim}}_Z(G, d,  \varepsilon,\mathbb{P})}{\log \frac{1}{\varepsilon}},\\
		\overline{\mathrm{umdim}}_Z(G,d,\mathbb{P})&:=\limsup _{\varepsilon \rightarrow 0}\frac{\overline{\mathrm{umdim}}_Z(G, d, \varepsilon,\mathbb{P})}{\log \frac{1}{\varepsilon}},\\
		\overline{\mathrm{lmdim}}_Z(G, d,\mathbb{P})&:=\limsup _{\varepsilon \rightarrow 0}\frac{\underline{\mathrm{lmdim}}_Z(G, d, \varepsilon,\mathbb{P})}{\log \frac{1}{\varepsilon}}.\\
	\end{aligned}
	$$

	\hspace{4mm}
	The quantities $\overline{\mathrm{mdim}}_{Z}(G,d,\mathbb{P})$, $ \overline{\mathrm{umdim}}_{Z}(G, d,\mathbb{P})$, $ \overline{\mathrm{lmdim}}_{Z}(G, d,\mathbb{P})$ are called the upper metric mean dimension, $u$-upper metric mean dimension, $l$-upper metric mean dimension  of free semigroup action $G$ with respect to $\mathbb{P}$ on the set $Z$, respectively. 
	\begin{remark} If $\sharp Y=1$, $G_1=\{f\}$, these quantities coincides with the upper metric mean dimension, $u$-upper metric mean dimension, $l$-upper metric mean dimension of $f$ with 0 potential on the set $Z$ defined by Cheng et al. \cite{MR4216094}, respectively.
	\end{remark}
	
	\subsection{Properties of the upper metric mean dimension of free semigroup actions for non-compact sets}
	Using the basic properties of the Carath\'eodory–Pesin dimension \cite{MR1489237} and definitions, we get the following basic properties of  upper metric mean dimension, $u$-upper metric mean dimension and $l$-upper metric mean dimension  of free semigroup actions for non-compact sets.

		\begin{Proposition} \label{prop1}
			Let $G$ be the free semigroup acting on $X$ generated by $G_1=\{f_y:\, y\in Y\}$. Then
		\begin{enumerate}[(i)]
			\item \label{item1}
			$\overline{\mathrm{mdim}}_{Z_1}(G, d, \mathbb{P}) \le \overline{\mathrm{mdim}}_{Z_2}(G, d, \mathbb{P})$,   $\overline{\mathrm{umdim}}_{Z_1}(G, d, \mathbb{P}) \le \overline{\mathrm{umdim}}_{Z_2}(G, d, \mathbb{P}), \\
			\overline{\mathrm{lmdim}}_{Z_1}(G, d, \mathbb{P}) \le \overline{\mathrm{lmdim}}_{Z_2}(G, d, \mathbb{P})$,  if $Z_1 \subset Z_2 \subset X$.
			
			\item \label{item2}
			$\overline{\mathrm{mdim}}_Z(G, d, \mathbb{P}) = \sup _{i \ge 1} \overline{\mathrm{mdim}}_{Z_i}(G, d, \mathbb{P}),  \\
			\overline{\mathrm{umdim}}_Z(G, d, \mathbb{P}) \ge \sup _{i \ge 1} \overline{\mathrm{umdim}}_{Z_i}(G, d, \mathbb{P}), \\
			\overline{\mathrm{lmdim}}_Z(G, d, \mathbb{P}) \ge \sup _{i \ge 1} \overline{\mathrm{lmdim}}_{Z_i}(G, d, \mathbb{P})$, if $Z=\bigcup_{i \ge 1} Z_i$.
			
			\item  \label{item3}
			$\overline{\mathrm{mdim}}_Z(G, d, \mathbb{P}) \le \overline{\mathrm{lmdim}}_Z(G,d, \mathbb{P}) \le \overline{\mathrm{umdim}}_Z(G, d, \mathbb{P})$ for any subset $Z \subset X$.
			
		\end{enumerate}
	\end{Proposition}

	\hspace{4mm}
	Similar to the Theorem 2.2 in \cite{MR1489237} and Lemma 3.2 in \cite{MR3918203},  we obtain the following result:
	\begin{Proposition}\label{Proposition1}
		For any subset $Z\subset X$, one has 
		$$
		\begin{aligned}
			\overline{\mathrm{lmdim}}_{Z}(G,d,\mathbb{P})&=\limsup_{\varepsilon\to 0}\liminf_{N \to \infty}\frac{\log \Lambda(Z, N, \varepsilon,G, d,\mathbb{P})}{N\log\frac{1}{\varepsilon}},\\
			\overline{\mathrm{umdim}}_{Z}(G, d,\mathbb{P})&=\limsup_{\varepsilon\to 0}\limsup_{N \to \infty}\frac{\log \Lambda(Z, N, \varepsilon,G, d,\mathbb{P})}{N\log\frac{1}{\varepsilon}}.\\
		\end{aligned}
		$$
	\end{Proposition}
	\begin{proof}
		We will prove the first equality; the second one can be proved in a similar fashion. It is enough to show that 
		$$
		\overline{\mathrm{lmdim}}_{Z}(\varepsilon,G,d,\mathbb{P})=\liminf_{N \to \infty}\frac{\log \Lambda(Z, N, \varepsilon,G, d,\mathbb{P})}{N}
		$$
		for any $0<\varepsilon<1$. This can be checked as follows. Put 
		$$
		\alpha=\overline{\mathrm{lmdim}}_{Z}(\varepsilon,G,d,\mathbb{P}),\quad
		\beta=\liminf_{N \to \infty}\frac{\log \Lambda(Z, N, \varepsilon,G, d,\mathbb{P})}{N}.
		$$
		Given $\gamma>0$, one can choose a sequence $N_j\to\infty$ such that
		$$
		0=\underline{r}(Z, \alpha+\gamma, \varepsilon,G, d,\mathbb{P})=\lim_{j\to\infty}R(Z, \alpha+\gamma, N_j, \varepsilon,G, d,\mathbb{P}).
		$$
		It follows that $R(Z, \alpha+\gamma, N_j, \varepsilon,G, d,\mathbb{P})<1$ for all sufficiently large $j$. Therefore, for such numbers $j$, 
		$$
		e^{-(\alpha+\gamma) (N_j+1)}\Lambda(Z, N_j, \varepsilon,G, d,\mathbb{P})<1.
		$$
		Moreover,
		$$
		\alpha+\gamma \geq \frac{\log \Lambda(Z, N_j, \varepsilon,G, d,\mathbb{P})}{N_j+1}.
		$$
		Therefore,
		$$
		\alpha+\gamma \geq \liminf_{N \rightarrow \infty} \frac{\log \Lambda(Z, N, \varepsilon,G, d,\mathbb{P})}{N}.
		$$
		Hence,
		\begin{equation}\label{prop1-1}
			\alpha \geq \beta-\gamma.
		\end{equation}
		
		\hspace{4mm}
		Let us now choose a sequence $N_j^{\prime}\to\infty$ such that
		$$
		\beta=\lim _{j \rightarrow \infty} \frac{\log \Lambda(Z, N_j^\prime, \varepsilon,G, d,\mathbb{P})}{N_j^{\prime}}.
		$$
		We have
		$$
		\lim_{j\to\infty}R(Z, \alpha-\gamma, N_j^\prime, \varepsilon,G, d,\mathbb{P})\ge \underline{r}(Z, \alpha-\gamma, \varepsilon,G, d,\mathbb{P})=\infty.
		$$
		This implies that $R(Z, \alpha-\gamma, N_j^\prime, \varepsilon,G, d,\mathbb{P})\ge1$ for all sufficiently large $j$. Therefore, for such $j$, 
		$$
		e^{-(\alpha-\gamma) (N_j^\prime+1)}\Lambda(Z,  N_j^\prime, \varepsilon,G, d, \mathbb{P})\ge1.
		$$
		and hence
		$$
		\alpha-\gamma \leq \frac{\log \Lambda(Z, N_j^\prime, \varepsilon,G, d, \mathbb{P})}{N_j^{\prime}+1}.
		$$
		Taking the limit as $j \rightarrow \infty$ we obtain that
		$$
		\alpha-\gamma \leq \liminf _{N \rightarrow \infty} \frac{\log \Lambda(Z, N, \varepsilon,G, d, \mathbb{P})}{N}=\beta,
		$$
		and consequently,
		\begin{equation}\label{prop1-2}
			\alpha \leq \beta+\gamma.
		\end{equation}
		Since $\gamma$ can be chosen arbitrarily small, the inequalities (\ref{prop1-1}) and (\ref{prop1-2}) imply that $\alpha=\beta$.
	\end{proof}

\hspace{4mm}
For the free semigroup $G$ acting on $X$ generated by $G_1=\{f_y\}_{y\in Y}$, a subset $Z\subset X$ is called $G$-invariant if $f_y^{-1}(Z)=Z$ for all $y\in Y$.  For an invariant set, similar to the topological entropy of a sing map \cite{MR1489237} and free semigroup actions \cite{MR3918203},  and the metric mean dimension \cite{MR4216094} of a sing map, we have the following theorem.
	\begin{Proposition}
		For any $G$-invariant subset $Z\subset X$, 
		$$
		\overline{\mathrm{lmdim}}_{Z}(G, d, \mathbb{P})=\overline{\mathrm{umdim}}_{Z}(G, d, \mathbb{P}).
		$$
	\end{Proposition}
	
	\begin{proof}
		 Fix $0<\varepsilon<1$, $\mathcal{U}\in C(\varepsilon)$, $p,q\in\mathbb{N}$ and $w^{(1)}\in Y^p$, $w^{(2)}\in Y^q$. We can choose two collections of strings $\mathcal{G}_{w^{(1)}} \subset \mathcal{W}_{p+1}(\mathcal{U})$ and $\mathcal{G}_{w^{(2)}} \subset \mathcal{W}_{q+1}(\mathcal{U})$ which cover $Z$. Supposing that $(w^{(1)},\mathbf{U})\in \mathcal{G}_{w^{(1)}}$, $\mathbf{U}=\left(U_0, U_1 \cdots, U_{p}\right)$ and $(w^{(2)}, \mathbf{V})\in \mathcal{G}_{w^{(2)}}$, $\mathbf{V}=\left(V_0, V_1, \cdots, V_{q}\right)$, we define
		$$
		\mathbf{U V}:=\left(U_0, U_1, \cdots, U_{p}, V_0, V_1, \cdots, V_{q}\right).
		$$
		Fixed $i\in Y$, consider
		$$
		\mathcal{G}_{w^{(1)}iw^{(2)}}:=\left\{\left(w^{(1)}iw^{(2)},\mathbf{U V}\right): \mathbf{U} \in \mathcal{G}_{w^{(1)}} , \mathbf{V} \in \mathcal{G}_{w^{(2)}} \right\} \subset \mathcal{W}_{p+q+2}(\mathcal{U}).
		$$
		Then
		$$
		X_{w^{(1)}iw^{(2)}}(\mathbf{U V})=X_{w^{(1)}}(\mathbf{U}) \cap(f_{\overline{w^{(1)} i}})^{-1}\left(X_{w^{(2)}}(\mathbf{V})\right).
		$$
		Since $Z$ is a $G$-invariant set, the collection of strings $\mathcal{G}_{w^{(1)}iw^{(2)}}$ also covers $Z$. By the definition of $\Lambda_{w^{(1)}iw^{(2)}}(Z, p+q+1, \varepsilon,G, d)$, we have
		$$
		\Lambda_{w^{(1)}iw^{(2)}}(Z, p+q+1, \varepsilon,G, d)\le \sharp\mathcal{G}_{w^{(1)}iw^{(2)}}\le \sharp\mathcal{G}_{w^{(1)}}\times\sharp\mathcal{G}_{w^{(2)}}.
		$$
		This implies that
		$$
		\Lambda_{w^{(1)}iw^{(2)}}(Z, p+q+1, \varepsilon,G, d) \le \Lambda_{w^{(1)}}(Z, p, \varepsilon,G, d)\times\Lambda_{w^{(2)}}(Z, q, \varepsilon,G, d) .
		$$
		Then,
		$$
		\begin{aligned}
			\Lambda(Z, p+q+1, \varepsilon,G, d, \mathbb{P})&=\int_{ Y^{\mathbb{N}}} \Lambda_{w(\omega)}(Z, p+q+1, \varepsilon,G, d) d\mathbb{P}(\omega)\\
			&\le\int_{ Y^{\mathbb{N}}} \Lambda_{w^{(1)}(\omega)}(Z, p, \varepsilon,G, d)\times\Lambda_{w^{(2)}(\sigma^{p+1}\omega)}(Z, q, \varepsilon,G, d)d\mathbb{P}\\
			&= \Lambda(Z, p, \varepsilon,G, d, \mathbb{P})\times\Lambda(Z, q, \varepsilon,G, d, \mathbb{P}).
		\end{aligned}
		$$
		Therefore, 
		$$
		\Lambda(Z, p+q+1, \varepsilon,G, d, \mathbb{P}) \le \Lambda(Z, p, \varepsilon,G, d, \mathbb{P})\times\Lambda(Z, q, \varepsilon,G, d, \mathbb{P}) .
		$$
		Let $a_p:=\log \Lambda(Z, p, \varepsilon,G, d, \mathbb{P})$. Note that 
		$\Lambda(Z, p, \varepsilon,G, d, \mathbb{P}) \ge1.$
		Therefore, $\inf _{p\ge 1} \frac{a_p}{p}  >-\infty$. So, by Theorem 4.9 of \cite{MR648108}, the limit $\lim _{p \rightarrow \infty} \frac{a_p}{p}$ exists and coincides with $\inf _{p \rightarrow \infty} \frac{a_p}{p}$.
	\end{proof}
	
	\hspace{4mm}
	Next, we discuss the relationship between the upper metric mean dimension and $u$-upper metric mean dimension of free semigroup action $G$ on $Z$ when $Z$ is a compact $G$-invariant set. 	Let $0<\varepsilon<1$ be given. We choose any $\lambda>\overline{\mathrm{mdim}}_Z(\varepsilon,G, d, \mathbb{P})$, then
	$$
	m(Z, \lambda, \varepsilon,G, d, \mathbb{P})=\lim _{N \rightarrow\infty} M(Z, \lambda, N, \varepsilon,G, d, \mathbb{P})=0.
	$$
	It is easy to check that
	$$
	\inf_{\mathcal{U} \in C(\varepsilon)}\lim _{N \rightarrow\infty} M(Z, \lambda, N, \mathcal{U},G, d, \mathbb{P})=0,
	$$
	where 
	$$
	M(Z, \lambda, N, \mathcal{U},G, d,\mathbb{P}):=\int_{Y^{\mathbb{N}}}  M_{w(\omega)}(Z, \lambda, N, \mathcal{U},G, d) d\mathbb{P}(\omega),
	$$
	$$
	M_w(Z, \lambda, N, \mathcal{U},G, d, \mathbb{P}):=\inf _{\mathcal{G}_w(\mathcal{U})}\left\{\sum_{(w_{\mathbf{U}}, \mathbf{U}) \in \mathcal{G}_w(\mathcal{U})} e^ { -\lambda \mathfrak{l}(\mathbf{U})} \right\},
	$$
	and the  infimum is taken over finite or countable collections of strings $\mathcal{G}_w(\mathcal{U}) \subset \mathcal{S}(\mathcal{U})$ such that $\mathfrak{l}(\mathbf{U})\ge N+1$ and $w_{\mathbf{U}}|_{[1,N]}=w$ for all $(w_{\mathbf{U}}, \mathbf{U})\in \mathcal{G}_w(\mathcal{U})$ and $Z \subset \bigcup_{(w_{\mathbf{U}}, \mathbf{U})\in \mathcal{G}_w(\mathcal{U})} X_{w_{\mathbf{U}}}(\mathbf{U}).$
	There exists a open cover $\mathcal{U}\in C(\varepsilon)$ such that 
	$$
	\lim _{N \rightarrow\infty} M(Z, \lambda, N, \mathcal{U},G, d, \mathbb{P})=0.
	$$
	Note that $M(Z, \lambda, N, \mathcal{U},G, d, \mathbb{P})$ is non-decreasing as $N$ increases and non-negative, it follows that $M(Z, \lambda, N, \mathcal{U},G, d, \mathbb{P})=0$ for all $N\in\mathbb{N}$. Hence, for any $N\in\mathbb{N}$, we have
	$$
	M_w(Z, \lambda, N, \mathcal{U},G, d)=0, \quad \nu ^N-a.e. \,\, w\in Y^N.
	$$
	Then there exists a finite or countable collections of strings $\mathcal{G}_w \subset \mathcal{S}(\mathcal{U})$ with $w_{\mathbf{U}}|_{[1,N]}=w$ and $\mathfrak{l}(\mathbf{U})\ge N+1$ for all $(w_{\mathbf{U}}, \mathbf{U})\in \mathcal{G}_w$ and 
	$
	Z \subset \bigcup_{(w_{\mathbf{U}}, \mathbf{U})\in \mathcal{G}_w} X_{w_{\mathbf{U}}}(\mathbf{U})
	$
	such that 
	\begin{equation}\label{equation1}
		Q(G,Z,\lambda, \mathcal{G}_w):=\sum_{(w_{\mathbf{U}}, \mathbf{U})\in\mathcal{G}_w }e^{-\lambda\mathfrak{l}(\mathbf{U})}<p<1.
	\end{equation}
	Since $Z$ is compact we can choose $\mathcal{G}_w$ to be finite and $K\ge 3$ to be a constant such that
	\begin{equation}\label{equation2}
		\mathcal{G}_w\subset\bigcup_{j=1}^{K}\mathcal{S}_j(\mathcal{U}).
	\end{equation}
	For any $w^{(1)},w^{(2)}\in Y^N$ and $i\in Y$, we can construct
	$$
	\mathcal{A}_{\mathcal{G}_{w^{(1)}}i\mathcal{G}_{w^{(2)}}}:=\left\{(w_{\mathbf{U}}iw_{\mathbf{V}}, \mathbf{U V}): (w_{\mathbf{U}},\mathbf{U})\in\mathcal{G}_{w^{(1)}},  (w_{\mathbf{V}},\mathbf{V})\in\mathcal{G}_{w^{(2)}} \right\}, 
	$$
	where $\mathcal{G}_{w^{(1)}}$ and $\mathcal{G}_{w^{(2)}}$ satisfy (\ref{equation1}), (\ref{equation2}). Then 
	$$
	X_{w_{\mathbf{U}}iw_{\mathbf{V}}}(\mathbf{U V})=X_{w_{\mathbf{U}}}(\mathbf{U})\cap\left( f_{\overline{w_{\mathbf{U}}i}}\right) ^{-1}(X_{w_{\mathbf{V}}}\left (\mathbf{V})\right ),
	$$
	where $\mathfrak{l}(\mathbf{U V})\ge 2(N+1)$. Since  $Z$ is $G$-invariant, then $	\mathcal{A}_{\mathcal{G}_{w^{(1)}}i\mathcal{G}_{w^{(2)}}}$ covers $Z$. It is easy to see that
	$$
	Q\left (G,Z,\lambda, \mathcal{A}_{\mathcal{G}_{w^{(1)}}i\mathcal{G}_{w^{(2)}}}\right )\le Q\left (G,Z,\lambda, \mathcal{G}_{w^{(1)}}\right )\times Q\left (G,Z,\lambda, \mathcal{G}_{w^{(2)}}\right )<p^2.
	$$
	By the induction, for each $n\in\mathbb{N}$, $w^{(1)},\cdots,w^{(n)}\in Y^N$ and $i_1,\cdots,i_{n-1}\in Y$, we can define $\mathcal{A}_{\mathcal{G}_{w^{(1)}} i_1 \mathcal{G}_{w^{(2)}}\cdots  i_{n-1} \mathcal{G}_{w^{(n)}}}$ which covers $Z$ and satisfies
	$$
	Q\left (G,Z,\lambda,\mathcal{A}_{\mathcal{G}_{w^{(1)}}i_1 \mathcal{G}_{w^{(2)}}\cdots i_{n-1} \mathcal{G}_{w^{(n)}}}\right )<p^n.
	$$
	Let $\Gamma_{\mathcal{G}_{w^{(1)}}i_1\mathcal{G}_{w^{(2)}}\cdots }:=\mathcal{A}_{\mathcal{G}_{w^{(1)}}}\cup\mathcal{A}_{\mathcal{G}_{w^{(1)}}i_1\mathcal{G}_{w^{(2)}}}\cup \cdots.$ Since $Z$ is $G$-invariant, then $\Gamma_{\mathcal{G}_{w^{(1)}}i_1\mathcal{G}_{w^{(2)}}\cdots}$ covers $Z$ and 
	$$
	Q\left (G,Z,\lambda,\Gamma_{\mathcal{G}_{w^{(1)}}i_1\mathcal{G}_{w^{(2)}}\cdots }\right )\le \sum_{n=1}^\infty p^n<\infty.
	$$
	Therefore, for any $w^{(j)}\in Y^N$ and 
	$i_j\in Y$, $j\in\mathbb{N}$, there exists 
	$\Gamma_{\mathcal{G}_{w^{(1)}} i_1 \mathcal{G}_{w^{(2)}}\cdots }$ covering $Z$ and 
	$Q(G,Z,\lambda,\Gamma_{\mathcal{G}_{w^{(1)}}i_1\mathcal{G}_{w^{(2)}}\cdots })<\infty$. Put
	$$
	\mathcal{F}:=\left\{\Gamma_{\mathcal{G}_{w^{(1)}}i_1\mathcal{G}_{w^{(2)}}\cdots }:w^{(j)}\in Y^N,  i_j\in Y, j, N\in\mathbb{N}\right\}.
	$$
	\begin{condition}\label{Condition1}
		For any $N>0$ and $\nu^N-a.e.\, w\in Y^N$, there exists $\Gamma_{\mathcal{G}_{w}i_1\mathcal{G}_{w^{(2)}}\cdots }\in\mathcal{F}$ such that $w_{\mathbf{U}}|_{[1,N]}=w$ and $N+1\le \mathfrak{l}(\mathbf{U})\le N+K$ for any $(w_{\mathbf{U}},\mathbf{U})\in \Gamma_{\mathcal{G}_{w}i_1\mathcal{G}_{w^{(2)}}\cdots }$, where $K$ is a constant as that in (\ref{equation2}).
	\end{condition}
	\begin{Proposition}
		Under the Condition \ref{Condition1}, for any $G$-invariant and compact subset $Z\subset X$,
		$$
		\overline{\mathrm{mdim}}_{Z}(G, d)=\overline{\mathrm{lmdim}}_{Z}(G, d)=\overline{\mathrm{umdim}}_{Z}(G, d).
		$$
	\end{Proposition}
	\begin{proof}
		Under Condition \ref{Condition1}. For any $N>0$ and $\nu^N-a.e.\, w\in Y^N$, there is $\Gamma_{\mathcal{G}_{w}i_1\mathcal{G}_{w^{(2)}}\cdots }\in\mathcal{F}$ covering $Z$ such that $w_{\mathbf{U}}|_{[1,N]}=w$ for any $(w_{\mathbf{U}}, \mathbf{U})\in\Gamma_{\mathcal{G}_{w}i_1\mathcal{G}_{w^{(2)}}\cdots }$. Then for any $x\in Z$, there exists a string $(w_{\mathbf{U}},\mathbf{U})\in\Gamma_{\mathcal{G}_{w}i_1\mathcal{G}_{w^{(2)}}\cdots }$, $\mathbf{U}=(U_0,U_1,\cdots,U_N,\cdots,U_{N+P})$, such that $x\in X_{w_{\mathbf{U}}}(\mathbf{U})$, where $0\le P\le K$. Let $U^*:=(U_0,U_1,\cdots,U_N)$. Then $X_{w_{\mathbf{U}}}(\mathbf{U})\subset X_w(\mathbf{U}^*)$. If $\Gamma_w^*$ denotes the collection of substrings $(w, \mathbf{U}^*)$ constructed above, then
		$$
		\begin{aligned}
			e^{-\lambda(N+1)}\Lambda_w(Z,N,\varepsilon,G,d)&\le e^{-\lambda(N+1)}\cdot \sharp\Gamma_w^*\\
			&\le \max\{1,e^{\lambda K}\}\cdot Q(G,Z,\lambda,\Gamma_{\mathcal{G}_{w}i_1\mathcal{G}_{w^{(2)}}\cdots})\\
			&\le \max\{1,e^{\lambda K}\}\cdot\sum_{n=1}^\infty p^n<\infty.
		\end{aligned}
		$$
		Therefore, 
		{
		$$
		\begin{aligned}
        R(Z,\lambda,N,\varepsilon,G,d, \mathbb{P})&=e^{-\lambda(N+1)}\Lambda(Z,N,\varepsilon,G,d, \mathbb{P})\\
			&=e^{-\lambda(N+1)}\int_{Y^{\mathbb{N}}} \Lambda_{w(\omega)}(Z,N,\varepsilon,G,d)d\mathbb{P}(\omega)<\infty.
		\end{aligned}
		$$
	}
		Then we have $\lambda>\overline{\mathrm{umdim}}_Z(\varepsilon,G,d, \mathbb{P})$. Hence,
		$$
		\overline{\mathrm{mdim}}_Z(\varepsilon,G,d, \mathbb{P})\ge \overline{\mathrm{umdim}}_Z(\varepsilon,G,d, \mathbb{P}).
		$$
		Dividing both sides of this inequality by $ \log\frac{1}{\varepsilon} $, and letting $\varepsilon\to 0$ to take the limitsup, we can get that
		$$
		\overline{\mathrm{mdim}}_Z(G,d, \mathbb{P})\ge \overline{\mathrm{umdim}}_Z(G,d, \mathbb{P}),
		$$
		as we wanted to prove.
	\end{proof}
	
	\subsection{Upper metric mean dimension of free semigroup actions for non-compact sets using open covers using Bowen balls}
	For $N\in\mathbb{N}$, $w\in Y^N$, $\lambda \in \mathbb{R}$, $Z \subset X$ and $0<\varepsilon<1$, we set
	$$
	M_w^B(Z, \lambda, N, \varepsilon,G, d):=\inf_{\Gamma_w}\left \{\sum_{i\in I}e^{-\lambda (|w_i|+1)}\right \},
	$$
	where the infimum is taken over all finite or countable collections $\Gamma_w=\left\{B_{w_i}\left(x_i, \varepsilon\right)\right\}_{i \in I}$ covering $Z$ with $|w_i| \ge N$ and $w_i|_{[1,N]}=w$.
		
	\hspace{4mm}
	For $\omega\in Y^\mathbb{N}$, put $w(\omega):=\omega|_{[1,N]}$, we define
	$$
	M^B(Z, \lambda, N, \varepsilon,G, d, \mathbb{P}):=\int_{Y^\mathbb{N}}M_{w(\omega)}^B(Z, \lambda, N, \varepsilon,G, d)d\mathbb{P}(\omega).
	$$
	Moreover, the function $N \mapsto M(Z, \lambda, N, \varepsilon,G, d)$ is non-decreasing as $N$ increases. Therefore, the following limit exists
	$$
	m^B(Z, \lambda, \varepsilon,G, d, \mathbb{P}):=\lim_{N\to\infty}M^B(Z, \lambda, N, \varepsilon,G, d, \mathbb{P}).
	$$
	
	\hspace{4mm}
	Similarly, we define
	$$
	\begin{aligned}
		R_w^B(Z, \lambda, N, \varepsilon,G, d):&=\inf_{\Gamma_w}\left \{\sum_{i\in I}e^{-\lambda (N+1)}\right \}\\
		&=e^{-\lambda (N+1)} \Lambda^B_w(Z, N, \varepsilon,G, d),
	\end{aligned}
	$$
	where $ \Lambda^B_w(Z, N, \varepsilon,G, d):=\inf_{\Gamma_w}\left \{\sharp\Gamma_w\right \},$
	and the infimum is taken over all finite or countable collections $\Gamma_w=\left\{B_{w}\left(x_i, \varepsilon\right)\right\}_{i \in I}$ covering $Z$ .
	
	\hspace{4mm}
	For $\omega\in Y^\mathbb{N}$, put $w(\omega):=\omega|_{[1,N]}$, we define
	$$
	\begin{aligned}
		R^B(Z, \lambda, N, \varepsilon,G, d,\mathbb{P}):&=\int_{Y^\mathbb{N}}R^B_{w(\omega)}(Z, \lambda, N, \varepsilon,G, d)d\mathbb{P}(\omega) \\
		&=e^{-\lambda (N+1)}\Lambda^B(Z, N, \varepsilon,G, d,\mathbb{P}),\\
	\end{aligned}
	$$
	where $\Lambda^B(Z, N, \varepsilon,G, d,\mathbb{P})=\int_{Y^\mathbb{N}}\Lambda^B_{w(\omega)}(Z, N, \varepsilon,G, d)d\mathbb{P}(\omega).$ We set
	$$
	\begin{aligned}
		\overline{r}^B(Z, \lambda, \varepsilon,G, d,\mathbb{P}):&=\limsup_{N\to\infty}R^B(Z, \lambda, N, \varepsilon,G, d,\mathbb{P}),\\
		\underline{r}^B(Z, \lambda, \varepsilon,G, d,\mathbb{P}):&=\liminf_{N\to\infty}R^B(Z, \lambda, N, \varepsilon,G, d,\mathbb{P}).
	\end{aligned}
	$$
	
	\hspace{4mm}
	It is readily to check that $m^B(Z, \lambda, \varepsilon,G, d,\mathbb{P}), \overline{r}^B(Z, \lambda, \varepsilon,G, d,\mathbb{P}), \underline{r}^B(Z, \lambda, \varepsilon,G, d,\mathbb{P})$ have a critical value of parameter $\lambda$ jumping from $\infty$ to 0 . We respectively denote their critical values as
	$$
	\begin{aligned}
		\overline{\mathrm{mdim}}_{Z}^B(\varepsilon,G, d,\mathbb{P}) :&=\inf \{\lambda: m^B(Z, \lambda, \varepsilon,G, d,\mathbb{P})=0\} \\
		&=\sup \{\lambda: m^B(Z, \lambda, \varepsilon,G, d,\mathbb{P})=\infty\}, \\
		\overline{\mathrm{umdim}}_{Z}^B(\varepsilon,G, d,\mathbb{P}) :&=\inf \{\lambda: \overline{r}^B(Z, \lambda, \varepsilon,G, d,\mathbb{P})=0\} \\
		&=\sup \{\lambda: \overline{r}^B(Z, \lambda, \varepsilon,G, d,\mathbb{P})=\infty\},\\
		\overline{\mathrm{lmdim}}_{Z}^B(\varepsilon,G, d,\mathbb{P}) :&=\inf \{\lambda: \underline{r}^B(Z, \lambda, \varepsilon,G, d,\mathbb{P})=0\} \\
		&=\sup \{\lambda: \underline{r}^B(Z, \lambda, \varepsilon,G, d,\mathbb{P})=\infty\}.
	\end{aligned}
	$$

	\hspace{4mm}
	
	\begin{thm}
		For any subset $Z\subset X$, one has
		$$
		\begin{aligned}
			\overline{\mathrm{mdim}}_{Z}(G, d,\mathbb{P})&=\limsup _{\varepsilon \rightarrow 0} \frac{\overline{\mathrm{mdim}}_{Z}^B(\varepsilon,G, d,\mathbb{P})}{\log \frac{1}{\varepsilon}}, \\
			\overline{\mathrm{umdim}}_{Z}(G, d,\mathbb{P})&=\limsup _{\varepsilon \rightarrow 0} \frac{\overline{\mathrm{umdim}}_{Z}^B(\varepsilon,G, d,\mathbb{P}) }{\log \frac{1}{\varepsilon}},\\
			\overline{\mathrm{lmdim}}_{Z}(G, d,\mathbb{P})&=\limsup _{\varepsilon \rightarrow 0} \frac{\overline{\mathrm{lmdim}}_{Z}^B(\varepsilon,G, d,\mathbb{P})}{\log \frac{1}{\varepsilon}}.
		\end{aligned}
		$$
	\end{thm}
	\begin{proof}
		We will prove the first equality; the second and third ones can be proved in a similar fashion.
		Let $\mathcal{U}$ be an open covers of $X$ with diameter less than $\varepsilon$, and $\delta(\mathcal{U})$ be the Lebesgue number of $\mathcal{U}$. It is easy to see that for every $x\in X$, if $x\in X_{w_{\mathbf{U}}}(\mathbf{U})$ for some $(w_{\mathbf{U}}, \mathbf{U})\in\mathcal{S}(\mathcal{U})$, then 
		\begin{equation}\label{equation8}
			B_{w_{\mathbf{U}}}(x,\frac{1}{2}\delta(\mathcal{U}))\subset X_{w_{\mathbf{U}}}(\mathbf{U})\subset B_{w_{\mathbf{U}}}(x,2\mathrm{diam}(\mathcal{U})).
		\end{equation}
		It follows that
		$$
		\begin{aligned}
			\inf _{\mathcal{G}_w(\mathcal{U})}\left\{\sum_{(w_{\mathbf{U}}, \mathbf{U}) \in \mathcal{G}_w(\mathcal{U})} e^ { -\lambda \mathfrak{l}(\mathbf{U})} \right\}
			&\ge M_w^B(Z, \lambda, N, 2\mathrm{diam} (\mathcal{U}),G, d)\\
			&\ge M_w^B(Z, \lambda, N, 2\varepsilon,G, d)
			,\quad \text{for any } \mathcal{U}\in C(\varepsilon).
		\end{aligned}
		$$
		Thus,
		\begin{equation}\label{equation9}
			M_w(Z, \lambda, N, \varepsilon,G, d)\ge M_w^B(Z, \lambda, N, 2\varepsilon,G, d).
		\end{equation}
		On the other hand, consider a open cover $\mathcal{U}:=\left\{B(x,\frac{\varepsilon}{2}): x\in X\right\}$. It is easy to check that $\frac{\varepsilon}{2}$ is a Lebesgue number of $\mathcal{U}$. It follows from (\ref{equation8}) that
		\begin{equation}\label{equation10}
			\begin{aligned}
				M_w^B(Z, \lambda, N, \frac{\varepsilon}{4},G, d)&\ge \inf _{\mathcal{G}_w(\mathcal{U})}\left\{\sum_{(w_{\mathbf{U}}, \mathbf{U}) \in \mathcal{G}_w(\mathcal{U})} e^ { -\lambda \mathfrak{l}(\mathbf{U})} \right\}\\
				&\ge M_w(Z, \lambda, N, \varepsilon,G, d).
			\end{aligned}
		\end{equation}
		We conclude by (\ref{equation9}) and (\ref{equation10}) that
		$$
		\overline{\mathrm{mdim}}_{Z}(G, d,\mathbb{P})=\limsup _{\varepsilon \rightarrow 0} \frac{\overline{\mathrm{mdim}}_{Z}^B(\varepsilon,G, d,\mathbb{P})}{\log \frac{1}{\varepsilon}}.
		$$
	\end{proof}
\begin{remark}
	\item [(\rmnum{1})]
         	If $\sharp Y=1$, then $\overline{\mathrm{mdim}}_{Z}(G, d,\mathbb{P})$ is equal to the upper mean metric dimension of a single map for non-compact subset $Z$ defined by Lima and Varandas \cite{MR4308163}.
	\item [(\rmnum{2})]
	If $\sharp Y=m$ and $\mathbb{P}$ is generated by the probability vector $\mathbf{p}:=\left(\frac{1}{m}, \cdots, \frac{1}{m}\right)$, then the critical values $\overline{\mathrm{umdim}}_Z(\varepsilon,G, d,\mathbb{P}),\overline{\mathrm{mdim}}_Z(\varepsilon,G, d,\mathbb{P}), \underline{\mathrm{lmdim}}_Z(\varepsilon,G, d,\mathbb{P})$ are equal to $h_Z(\varepsilon,G), \overline{\mathrm{Ch}}_Z(\varepsilon,G),\underline{\mathrm{Ch}}_Z(\varepsilon,G)$, respectively, as defined by Ju et al. \cite{MR3918203}.
	\item [(\rmnum{3})] 
	Ghys, Langevin and Walczak proposed in \cite{MR926526} the topological entropy of a semigroup action $G$ which differs from the way Bufetov \cite{MR1681003} was defined.  For $n \in \mathbb{N}$, let
	$$
	B_n^G(x, \varepsilon):=\left\{y \in X: d(f_{\overline{w}}(x), f_{\overline{w}}(y))<\varepsilon \text { for all } w\in Y^i, 0 \leq i \leq n\right\},
	$$
	called the $n t h$-dynamical ball of center $x$ and radius $\varepsilon$ (see \cite{MR3137474,MR926526} for more details). Rodrigues et al. \cite{WOS:000899898700001} introduced the metric mean dimension of free semigroup actions for non-compact sets in the GLW setting. For all $w\in Y^n$, we have 
	$$
	B_n^G(x, \varepsilon)\subset B_w(x,\varepsilon).
	$$
	Thus, the  metric mean dimension here defined is a lower bound for the dimension given in Rodrigues et al. \cite{WOS:000899898700001}.
\end{remark}

\section{The proofs of main results}\label{proofs}
\subsection{The proof of Theorem \ref{theorem1}}
In this subsection, we obtain lower and upper estimations of the upper metric mean dimension of free semigroup action $G$ generated by $G_1=\{f_y:y\in Y\}$ using local metric mean dimensions.
\begin{proof}[Proof of Theorem \ref{theorem1} (\ref{theorem5.1})]
	Fix $\gamma>0$. For each $k\ge 1$, put
	$$
	Z_k:=\left\{x\in Z: \liminf_{n\to\infty}\frac{-\log\sup_{w\in Y^n}\left \{\mu\left (B_w(x,\varepsilon)\right )\right \}}{(n+1)\log\frac{1}{\varepsilon}}>s-\gamma\text{ for all }\varepsilon\in \left(0,\frac{1}{k}\right) \right\}.
	$$
	Since $\underline{\mathrm{mdim}}_\mu (x,G)\ge s$ for all $x\in Z$, the sequence $\left\{Z_k\right\}_{k=1}^{\infty}$ increases to $Z$. So by the continuity of the measure, we have
	$$
	\lim _{k \to \infty} \mu\left(Z_k\right)=\mu(Z)>0 .
	$$
	Then fix some $k_0 \ge 1$ with $\mu(Z_{k_0})>\frac{1}{2} \mu(Z)$. For each $N \ge 1$, put
	$$
	\begin{aligned}
		Z_{k_0,N}:=&\left\{ x\in Z_{k_0}:\liminf_{n\to\infty}\frac{-\log\sup_{w\in Y^n}\left \{\mu\left (B_w(x,\varepsilon)\right )\right \}}{(n+1)\log\frac{1}{\varepsilon}}>s+\gamma   \right. \\
		&\quad\quad\quad\quad\quad\quad\quad\quad\quad\quad\quad\quad\quad\left.
		\text { for all } n \ge N \text { and } \varepsilon \in\left(0, \frac{1}{k_0}\right)
		\right\}.
	\end{aligned}
	$$
	Since the sequence $\left\{Z_{k_0, N}\right\}_{N=1}^{\infty}$ increases to $Z_{k_0}$, we may pick an $N^* \geq 1$ such that $\mu\left(Z_{k_0, N^*}\right)>\frac{1}{2} \mu\left(Z_{k_0}\right)$. Write $Z^*=Z_{k_0, N^*}$ and $\varepsilon^*=\frac{1}{k_0}$. Then $\mu(Z^*)>0$ and
	\begin{equation}\label{formula4-1}
		\sup_{w\in Y^n}\left \{\mu\left (B_w(x,\varepsilon)\right )\right \}<e^{-(s-\gamma)(n+1)\log\frac{1}{\varepsilon}}
	\end{equation}
for all $x \in Z^*$, $0<\varepsilon \leq\varepsilon^*$ and $n \ge N^*$.
	For any $N\ge N^*$ and $w\in Y^N$, set a countable cover of  $Z^*$
	$$
	\mathcal{F}_w:=\left\{B_{w^{(i)}}\left (x_i,\frac{\varepsilon}{2}\right ): w^{(i)}\in Y^{N^\prime},\,  N^\prime\ge N \text{ and } w^{(i)}|_{[1,N]}= w\right\},
	$$
	which satisfies
	$$
	Z^* \cap B_{w^{(i)}}\left (x_i,\frac{\varepsilon}{2}\right ) \neq \emptyset, \text { for all } i \geq 1 \text { and } 0<\varepsilon \le \varepsilon^* \text {. }
	$$
	For each $i \ge 1$, there exists an $y_i \in Z^* \cap B_{w^{(i)}}\left (x_i,\frac{\varepsilon}{2}\right )$. By the triangle inequality
	$$
	B_{w^{(i)}}\left (x_i,\frac{\varepsilon}{2}\right )\subset B_{w^{(i)}}\left (y_i,\varepsilon\right ).
	$$
	In combination with (\ref{formula4-1}), this implies
	$$
	\sum_{i \geq1} \mathrm{e}^{-(s-\gamma) (|w^{(i)}|+1)\log\frac{1}{\varepsilon}} \geq \sum_{i \geq 1} \mu\left(B_{w^{(i)}}\left (y_i,\varepsilon\right )\right) \geq \mu\left(Z^*\right)>0.
	$$
	Therefore, 
	$$
	M_w^B\left (Z^*, (s-\gamma)\log\frac{1}{\varepsilon}, N, \varepsilon,G, d\right )\ge \mu\left(Z^*\right)>0,
	$$
	for all $w\in F^N$ with $N \ge N^*$.
	Then 
	$$M^B\left (Z^*, (s-\gamma)\log\frac{1}{\varepsilon}, N, \varepsilon,G, d,\mathbb{P}\right )\ge \mu\left(Z^*\right)>0,
	$$
	and consequently
	$$
	m^B\left (Z^*, (s-\gamma)\log\frac{1}{\varepsilon}, \varepsilon,G, d,\mathbb{P}\right )=\lim_{N\to\infty}M^B\left (Z^*, (s-\gamma)\log\frac{1}{\varepsilon}, N, \varepsilon,G, d,\mathbb{P}\right )>0,
	$$
	which in turn implies that $\overline{\mathrm{mdim}}_{Z^*}^B(\varepsilon,G, d,\mathbb{P})\ge (s-\gamma)\log\frac{1}{\varepsilon}$. Dividing both sides of this inequality by $ \log\frac{1}{\varepsilon} $, and letting $\varepsilon\to 0$ to take the limitsup, we can get that 
	$$
	\overline{\mathrm{mdim}}_{Z^*}(G, d,\mathbb{P})\ge s-\gamma.
	$$
	Hence $\overline{\mathrm{mdim}}_{Z}(G, d,\mathbb{P})\ge s$ since  $\overline{\mathrm{mdim}}_{Z}(G, d,\mathbb{P})\ge \overline{\mathrm{mdim}}_{Z^*}(G, d,\mathbb{P})$ and $\gamma$ is arbitrary. The proof is completed now.
\end{proof}

\hspace{4mm}
First, we need the following lemma, which is much like the classical covering lemma, to prove Theorem \ref{theorem1} (\ref{theorem3}), and  the proof follows \cite{MR2412786} and is omitted.
\begin{lem}\cite{MR3918203}\label{coverlemma}
	Let $\varepsilon>0$ and $\mathcal{B}(\varepsilon):=\left\{B_w(x,\varepsilon): x \in X, w\in Y^N, N\in\mathbb{N}\right\}$. For any family $\mathcal{F} \subset \mathcal{B}(\varepsilon)$, there exists a (not necessarily countable) subfamily $\mathcal{G} \subset \mathcal{F}$ consisting of disjoint balls such that
	$$
	\bigcup_{B \in \mathcal{F}} B \subset \bigcup_{B_w(x, \varepsilon)\in \mathcal{G}} B_w\left (x, 3\varepsilon\right ).
	$$
\end{lem}

\begin{proof}[Proof of Theorem \ref{theorem1} (\ref{theorem3})]
	Since $\overline{\mathrm{mdim}}_\mu(x,G)\le s$ for all $x\in Z$, then for all $\omega\in Y^\mathbb{N}$ and $x\in Z$,
	$$
	\limsup_{\varepsilon\to 0}\liminf_{n\to\infty}\frac{-\log \mu\left( B_{\omega|_{[1,n]}}(x, \varepsilon)\right)}{(n+1)\log\frac{1}{\varepsilon}}\le\overline{\mathrm{mdim}}_\mu(x,G)\le s.
	$$
	Fixed $\gamma>0$, $N\in\mathbb{N}$ and $w\in Y^N$, we have $Z=\bigcup_{k\ge 1}Z_k$ where
	$$
	\begin{aligned}
		Z_k:=&\left\{ x\in Z:\liminf_{n\to\infty}\frac{-\log \mu\left(B_{\omega|_{[1,n]}}(x,\varepsilon)\right)}{(n+1)\log\frac{1}{\varepsilon}}<s+\gamma   \right. \\
		&\quad\quad\quad\quad\left.
		\text{for all }\varepsilon\in \left(0,\frac{1}{k}\right)\text{ for some } \omega\in Y^\mathbb{N}\text{ with }\omega|_{[1,N]}=w
		\right\}.
	\end{aligned}
	$$
	Now fix $k\ge 1$ and $0<\varepsilon<\frac{1}{3k}$. For each $x\in Z_k$, there exist $\omega_x\in Y^\mathbb{N}$ with $\omega_x|_{[1,N]}=w$ and a strictly increasing sequence $\{n_j^x\}_{j=1}^\infty$ such that
	$$
	\mu \left (B_{\omega_x|_{[1,n_j^x]}}(x,\varepsilon)\right )\ge e^{-(n_j^x+1)(s+\gamma)\log\frac{1}{\varepsilon}},\quad\text{for all } j\ge 1.
	$$
	So, the set $Z_k$ is contained in the union of the sets in the family
	$$
	\mathcal{F}_w:=\left\{ B_{\omega_x|_{[1,n_j^x]}}(x,\varepsilon): x\in Z_k,\omega_x\in Y^\mathbb{N},  \omega_x|_{[1,N]}=w, n_j^x\ge N\right\}.
	$$
	By Lemma \ref{coverlemma}, there exists a subfamily $\mathcal{G}_w=\{B_{\omega_{x_j}|_{[1,n_j]}}(x_j,\varepsilon)\}_{j\in J}\subset \mathcal{F}_w$ consisting of disjoint balls such that for all $j\in J$
	$$
	Z_k\subset \bigcup_{j\in J}B_{\omega_{x_j}|_{[1,n_j]}}(x_j,3\varepsilon)
	$$
	and
	$$
	\mu\left(B_{\omega_{x_j}|_{[1,n_j]}}(x_j,\varepsilon)\right) \ge e^{-(n_j+1)(s+\gamma)\log\frac{1}{\varepsilon}},\quad\text{for all } j\in J.
	$$
	The index set $J$ is at most countable since $\mu$ is a probability measure and $\mathcal{G}$ is a disjointed family of sets, each of which has a positive $\mu$-measure. Therefore, 
	$$
	\begin{aligned}
		M_w^B\left (Z_k, (s+\gamma)\log\frac{1}{\varepsilon}, N, 3\varepsilon,G, d,\mathbb{P}\right )&\le\sum_{j\in J} e^{-(n_j+1)(s+\gamma)\log\frac{1}{\varepsilon}}\\
		&\le \sum_{j\in J} \mu\left(B_{\omega_{x_j}|_{[1,n_j]}}(x_j,\varepsilon)\right) \le1,
	\end{aligned}
	$$
	where the disjointness of $\{B_{\omega_{x_j}|_{[1,n_j]}}(x_j,\varepsilon)\}_{j\in J}$ is used in the last inequality. It follows that 
	$$
	M^B\left (Z_k, (s+\gamma)\log\frac{1}{\varepsilon}, N, 3\varepsilon,G, d,\mathbb{P}\right )\le 1,
	$$ 
	and consequently
	$$
	m^B\left (Z_k, (s+\gamma)\log\frac{1}{\varepsilon}, 3\varepsilon,G, d,\mathbb{P}\right )=\lim_{N\to\infty}M^B\left (Z_k, (s+\gamma)\log\frac{1}{\varepsilon}, N, 3\varepsilon,G, d,\mathbb{P}\right )\le 1,
	$$
	which in turn implies that $\overline{\mathrm{mdim}}_{Z_k}^B(3\varepsilon,G,d,\mathbb{P})\le (s+\gamma)\log\frac{1}{\varepsilon}$ for any $0<\varepsilon<\frac{1}{3k}$. Dividing both sides of this inequality by $ \log\frac{1}{\varepsilon} $, and letting $\varepsilon\to 0$ to take the limitsup, we can get that
	$$
	\overline{\mathrm{mdim}}_{Z_k}(G,d,\mathbb{P})\le s+\gamma.
	$$ 
	As the arbitrariness of $\gamma$, we obtain that 
	$$
	\overline{\mathrm{mdim}}_{Z_k}(G,d,\mathbb{P})\le s, \quad\text{ for all }k\ge 1.
	$$
	By Proposition \ref{prop1} (\ref{item2}),
	$$
	\overline{\mathrm{mdim}}_{Z}(G,d,\mathbb{P})=\sup_{k\ge 1}\overline{\mathrm{mdim}}_{Z_k}(G,d,\mathbb{P})\le s.
	$$
	This finishes the proof of the theorem.
\end{proof}

\subsection{The Proof of Theorem \ref{theorem4}}
In the subsection, our purpose is to find the relationship between the $u$-upper metric mean dimension of free semigroup action $G$ generated by $G_1=\{f_y:\, y\in Y\}$ and the $u$-upper metric mean dimension of the corresponding skew product transformation $F$.

\hspace{4mm}
For  $Z\subset X$, $w\in Y^n$, and $0<\varepsilon<1$,  since $\Lambda_w^B(Z,n,\varepsilon,G,d)=r(w,\varepsilon,Z, G)$, then 
$$
\Lambda^B(Z,n,\varepsilon,G,d,\mathbb{P})=\int_{Y^\mathbb{N}} r(\omega|_{[1,n]},\varepsilon,Z, G) d\mathbb{P}(\omega).
$$
Therefore,
$$
\begin{aligned}
	\overline{\mathrm{umdim}}_Z(G,d,\mathbb{P})&=\limsup_{\varepsilon\to 0}\limsup_{n \rightarrow \infty}\frac{\log\int_{Y^\mathbb{N}} r(\omega|_{[1,n]},Z, \varepsilon,G) d\mathbb{P}(\omega)}{n\log\frac{1}{\varepsilon}}\\
	&=\limsup_{\varepsilon\to 0}\limsup_{n \rightarrow \infty}\frac{\log\int_{Y^\mathbb{N}} s(\omega|_{[1,n]},\varepsilon,Z, G) d\mathbb{P}(\omega)}{n\log\frac{1}{\varepsilon}}.
\end{aligned}
$$
and
$$
\begin{aligned}
	\overline{\mathrm{lmdim}}_Z(G,d,\mathbb{P})&=\limsup_{\varepsilon\to 0}\liminf_{n \rightarrow \infty}\frac{\log\int_{Y^\mathbb{N}} r(\omega|_{[1,n]},\varepsilon,Z, G) d\mathbb{P}(\omega)}{n\log\frac{1}{\varepsilon}}\\
	&=\limsup_{\varepsilon\to 0}\liminf_{n \rightarrow \infty}\frac{\log\int_{Y^\mathbb{N}} s(\omega|_{[1,n]},\varepsilon,Z, G) d\mathbb{P}(\omega)}{n\log\frac{1}{\varepsilon}}.
\end{aligned}
$$
\begin{remark}
	If $Z=X$, then $$\overline{\mathrm{umdim}}_X(G,d,\mathbb{P})=\overline{\mathrm{lmdim}}_X(G,d,\mathbb{P})=\overline{\mathrm{mdim}}_M(X,G,d,\mathbb{P}).$$
	Hence, the $u$-upper metric mean dimension and $l$-upper metric mean dimension of free semigroup actions on $X$ coincide with the upper metric mean dimension of free semigroup actions on $X$ defined by Carvalho et al.  \cite{MR4348410}.
\end{remark}

\hspace{4mm}
To prove Theorem \ref{theorem4}, we give the following two lemmas. The proof of these two lemmas is similar to that of Carvalho et al. \cite{MR4348410}. Therefore, we omit the proof.
\begin{lem}
	For any subset $Z\subset X$, if $\overline{\mathrm{dim}}_BY<\infty$ and $\nu\in \mathcal{H}_Y$, then
	$$\overline{\mathrm{dim}}_B(\mathrm{supp}\nu)+\overline{\mathrm{umdim}}_Z \left(G, d, \nu ^\mathbb{N}\right) \le \overline{\mathrm{umdim}}_{Y^{\mathbb{N}}\times Z}\left( F, D\right).$$
\end{lem}

\hspace{4mm}
Therefore, we establish the proof of Theorem \ref{theorem4} (\ref{theorem4-1}).  Furthermore, we can get the following more general than Theorem \ref{theorem4} (\ref{theorem4-2}) similar to Proposition 4.2 in  \cite{MR4348410}.
\begin{lem}\label{lem1}
	For any subset $Z\subset X$, if $\overline{\mathrm{dim}}_BY<\infty$ and $\nu \in \mathcal{H}_Y$, then
	$$
	\overline{\mathrm{dim}}_B(\operatorname{supp} \nu)+\overline{\mathrm{umdim}}_Z\left(G, d,\nu^\mathbb{N}\right)=\overline{\mathrm{umdim}}_{(\mathrm{supp} \nu )^{\mathbb{N}}\times Z}\left(F, D\right)
	$$
\end{lem}

\subsection{The proof of Theorem \ref{theorem5}}
In order to prove Theorem \ref{theorem5}, we first need the following auxiliary result.
\begin{lem}\label{lemma1}
	If $G$ satisfies the gluing orbit property, then the skew product map $F$ corresponding to the maps $G_1=\{f_y:y\in Y\}$ has the gluing orbit property.
\end{lem}
\begin{proof}
	Let $\varepsilon>0$ and set $\delta=\frac{\varepsilon}{2\cdot\mathrm{diam} Y}$. Then $\sum_{j>\lceil -\log\delta\rceil}2^{-j}\cdot \mathrm{diam} Y <\varepsilon$. Let $\mathfrak{p}_G(\delta)$ be the positive integer in the definition of gluing orbit property of $G$ (see Definition \ref{definition5}). 
	Let $(\omega_1,x_1),\cdots,(\omega_k,x_k)\in Y^{\mathbb{N}}\times X$ and $n_1,\cdots,n_k$ be given. 	Let 
	$$
	\begin{aligned}
		w_{(n_1)}&:=\omega_1|_{[1,\, n_1+\lceil -\log\delta\rceil]},\\
		w_{(n_2)}&:=\omega_2|_{[1,\, n_2+\lceil -\log\delta\rceil]},\\
		&\cdots\\
		w_{(n_k)}&:=\omega_k|_{[1,\, n_k+\lceil -\log\delta\rceil]}.
	\end{aligned}
	$$
	Assume that the integers $0\le p_1^G,\cdots,p_{k-1}^G\le \mathfrak{p}_G(\delta)$ satisfy the Definition \ref{definition5}. Pick
	$$
	\begin{aligned}
		w_{(p_1^G)}&:=\omega_1|_{[n_1+\lceil -\log\delta\rceil+1,\, n_1+\lceil -\log\delta\rceil+p_1^G]},\\
		w_{(p^G_2)}&:=\omega_2|_{[n_2+\lceil -\log\delta\rceil+1,\, n_2+\lceil -\log\delta\rceil+p_2^G]},\\
		&\cdots\\
		w_{(p_{k-1}^G)}&:=\omega_{k-1}|_{[n_{k-1}+\lceil -\log\delta\rceil+1, \, n_{k-1}+\lceil -\log\delta\rceil+p_{k-1}^G]}.
	\end{aligned}
	$$
	Then there exists $y\in X$ such that $d_{w_{(n_1)}}(y,x_1)\le\delta$ and 
	$$
	d_{w_{(n_j)}}(f_{\overline{w_{(p_{j-1}^G)}w_{(n_{j-1})}\cdots w_{(p_1^G)}w_{(n_1)}}}(y),x_j)\le \delta,\quad\text{for all }2\le j\le k.
	$$
	Consider
	$$
	\omega:=w_{(n_1)}w_{(p_1^G)}w_{(n_2)}w_{(p_2^G)}\cdots w_{(n_{k-1})}w_{(p_{k-1}^G)}w_{(n_k)}\cdots\in Y^\mathbb{N}.
	$$
	Let $\mathfrak{p}_F(\varepsilon):=\mathfrak{p}_G(\delta)+\lceil -\log\delta\rceil$ and 
	$$
	p^F_1:=p_1^G+\lceil -\log\delta\rceil,\, p^F_2:=p_2^G+\lceil -\log\delta\rceil, \cdots,p^F_{k-1}:=p_{k-1}^G+\lceil -\log\delta\rceil.
	$$
	It is easy to see that $0\le p_1^F,p_2^F,\cdots, p_{k-1}^F\le\mathfrak{p}_F(\varepsilon)$. Hence,
	$$
	\begin{aligned}
		D\left( F^i(\omega,y),F^i(\omega_1,x_1)\right)&=\max\left\{ d^\prime (\sigma^i(\omega),\sigma^i(\omega_1)), d(f_{\overline{\omega|_{[1,i]}}}(y),f_{\overline{\omega_1|_{[1,i]}}}(x_1))\right\}\\
		&=\max\left\{ d^\prime (\sigma^i(\omega),\sigma^i(\omega_1)), d(f_{\overline{w_{(1)}|_{[1,i]}}}(y),f_{\overline{w_{(1)}|_{[1,i]}}}(x_1))\right\}\\
		&\le\varepsilon,
	\end{aligned}
	$$
	for all $0\le i\le n_1-1$, and 
	$$
	\begin{aligned}
		&D\left( F^{M_{j-1}+i}(\omega,y),F^i(\omega_j,x_j)\right)\\
		=&\max\left\{ d^\prime (\sigma^{M_{j-1}+i}(\omega),\sigma^i(\omega_j)),d(f_{\overline{\omega|_{[1,M_{j-1}+i]}}}(y),f_{\overline{\omega_j|_{[1,i]}}}(x_j))  \right\}\\
		=&\max\left\{ d^\prime (\sigma^{M_{j-1}+i}(\omega),\sigma^i(\omega_j)),d(f_{\overline{w_{(j)}|_{[1,i]}w_{(p_{j-1}^G)}w_{(j-1)}\cdots w_{(p_1^G)}w_{(1)}}}(y),f_{\overline{w_{(j)}|_{[1,i]}}}(x_j))  \right\}\\
		\le &\varepsilon,
	\end{aligned}
	$$
	where $M_{j-1}:= n_1+p_1^F+\cdots +n_{j-1}+p^F_{j-1}$, for all $2\le j \le k$ and $0\le i\le n_j-1$.
\end{proof}
\begin{proof}[Proof of Theorem \ref{theorem5}]
	Suppose $F: Y^{\mathbb{N}} \times X \rightarrow Y^{\mathbb{N}} \times X$ is the skew product transformation corresponding to the maps $G_1=\{f_y:y\in Y\}$. Define a function $\psi: Y^\mathbb{N} \times X \rightarrow \mathbb{R}^d$ such that for any $\omega=\left(i_1 i_2 \cdots\right) \in Y^\mathbb{N}$, the map $\psi$ satisfies $\psi(\omega, x)=\varphi(x)$, then $\psi \in C\left(Y^\mathbb{N} \times X, \mathbb{R}^d\right)$. Let
	$$
	I_\psi(F):=\left\{(\omega, x) \in Y^\mathbb{N} \times X: \lim _{n \rightarrow \infty} \frac{1}{n} \sum_{j=0}^{n-1} \psi\left(F^j(\omega, x)\right) \text { does not exist }\right\}.
	$$
	From Lemma \ref{lemma1}, $F$ has the gluing orbit property. In \cite{MR4308163}, the authors proved the following result:
	\begin{center}
		(a) either $I_\psi(F)=\emptyset$; or (b) $\overline{\mathrm{mdim}}_{I_\psi(F)}(F,D)=\overline{\mathrm{mdim}}_M(Y^\mathbb{N} \times X, F,D)$.
	\end{center}
	We just consider the case of $I_{\varphi}(G) \neq \emptyset$. For any $(\omega,x) \in Y^\mathbb{N}\times X$, it's easy to see that
	$$
	\psi\left(F^j(\omega, x)\right)=\varphi\left(f_{\overline{\omega|_{[1,j]}}}(x)\right),
	$$
	then we have
	\begin{equation}\label{equation4}
		\frac{1}{n} \sum_{j=0}^{n-1} \psi\left(F^j(\omega, x)\right)=\frac{1}{n} \sum_{j=0}^{n-1} \varphi\left(f_{\overline{\omega|_{[1,j]}}}(x)\right).
	\end{equation}
	
	For any $(\omega, x) \in I_\psi(F)$ and by (\ref{equation4}), we obtain that if 
	$$
	\lim _{n \rightarrow \infty}\frac{1}{n} \sum_{j=0}^{n-1} \varphi\left(f_{\overline{\omega|_{[1,j]}}}(x)\right)
	$$ 
	does not exist, then $x \in I_{\varphi}(G)$. So $(\omega, x) \in Y^\mathbb{N} \times I_{\varphi}(G)$. It implies that
	$$
	I_\psi(F) \subseteq Y^\mathbb{N} \times I_{\varphi}(G) \subseteq Y^\mathbb{N} \times X
	$$
	By Theorem \ref{theorem4}, we get
	\begin{equation}\label{equation6}
		\overline{\mathrm{umdim}}_{Y^\mathbb{N}\times I_{\varphi}(G)}(F,D)=\overline{\mathrm{umdim}}_{ I_{\varphi}(G)}(G,d,\mathbb{P}).
	\end{equation}
	Since $I_\psi(F) \subseteq Y^\mathbb{N} \times I_{\varphi}(G)$ and (b), we get
	\begin{equation}\label{equation7}
		\begin{aligned}
			\overline{\mathrm{mdim}}_M(Y^\mathbb{N}\times X, F,D)=\overline{\mathrm{mdim}}_{I_\psi(F)}(F,D)&\le \overline{\mathrm{mdim}}_{Y^\mathbb{N}\times I_\psi(F)}(F,D)\\
			&\le \overline{\mathrm{umdim}}_{Y^\mathbb{N}\times I_\psi(F)}(F,D).
		\end{aligned}
	\end{equation}
	From (\ref{equation5}), (\ref{equation6}) and (\ref{equation7}), we have
	$$
	\begin{aligned}
		\overline{\mathrm{dim}}_B Y +\overline{\mathrm{mdim}}_M(X,G,d, \mathbb{P})&=\overline{\mathrm{mdim}}_M(Y^\mathbb{N} \times X, F,D) \\
		&=\overline{\mathrm{mdim}}_{I_\psi(F)}(F,D)\\
		&\le \overline{\mathrm{umdim}}_{Y^\mathbb{N}\times I_\varphi(G)}(F,D)\\
		&= \overline{\mathrm{dim}}_B Y +\overline{\mathrm{umdim}}_{ I_{\varphi}(G)}(G,d,\mathbb{P}),
	\end{aligned}
	$$
	then
	$$
	\overline{\mathrm{umdim}}_{ X}(G,d,\mathbb{P})= \overline{\mathrm{mdim}}_M(X,G,d, \mathbb{P})\le \overline{\mathrm{umdim}}_{ I_{\varphi}(G)}(G,d,\mathbb{P}).
	$$
	Obviously,
	$$
	\overline{\mathrm{umdim}}_{ I_{\varphi}(G)}(G,d,\mathbb{P})\le\overline{\mathrm{mdim}}_M(X,G,d, \mathbb{P})=  \overline{\mathrm{umdim}}_{ X}(G,d,\mathbb{P}).
	$$
	Hence,
	$$
	\overline{\mathrm{umdim}}_{ I_{\varphi}(G)}(G,d,\mathbb{P})=\overline{\mathrm{mdim}}_M(X,G,d, \mathbb{P})=  \overline{\mathrm{umdim}}_{ X}(G,d,\mathbb{P}).
	$$
\end{proof}

\begin{remark}
	A similar result of the upper metric mean dimension of free semigroup actions for the whale phase in the GLW setting was obtained by Rodrigues et al. \cite{WOS:000899898700001}.
\end{remark}

	\bibliographystyle{abbrv}
	\bibliography{references-metricmeandimensions.bib} 
	
\end{document}